\documentclass[10pt,a4paper,english]{amsart}
\usepackage[latin9]{inputenc}
\usepackage{color}
\usepackage{babel}
\usepackage{verbatim}
\usepackage{amsthm}
\usepackage{amstext}
\usepackage{amssymb}
\usepackage{accents}
\usepackage{esint}
\usepackage[unicode=true,pdfusetitle,
 bookmarks=true,bookmarksnumbered=false,bookmarksopen=false,
 breaklinks=false,pdfborder={0 0 1},backref=false,colorlinks=false]
 {hyperref}

\makeatletter

\pdfpageheight\paperheight
\pdfpagewidth\paperwidth

\numberwithin{equation}{section}
\numberwithin{figure}{section}
\theoremstyle{plain}
\newtheorem{thm}{\protect\theoremname}
  \theoremstyle{definition}
  \numberwithin{thm}{section}
  \newtheorem{defn}[thm]{\protect\definitionname}
  \theoremstyle{plain}
  \newtheorem{prop}[thm]{\protect\propositionname}
  \theoremstyle{definition}
  \newtheorem{example}[thm]{\protect\examplename}
  \newtheorem{rem}[thm]{\protect\remarkname}
\newtheorem*{thm*}{\protect\theoremname}
  \theoremstyle{definition}

\usepackage{amsaddr}\usepackage{amsthm}\usepackage{amsfonts}\usepackage{amscd}\usepackage{url}\usepackage{color}\usepackage[matrix,arrow]{xy}\usepackage{mathrsfs}\usepackage{enumerate}\usepackage{amsopn}
\usepackage{bbm}
\usepackage{cite}\allowdisplaybreaks[1]
\usepackage{bbm}\usepackage{stmaryrd}

\usepackage{mdef}

\title[Random attractors for locally monotone SPDE]{{Random attractors for locally monotone stochastic partial differential equations}}
\author[B. Gess]{Benjamin Gess}
\email{benjamin.gess@mis.mpg.de}
\address{ Fakult\"at f\"ur Mathematik, Universit\"at Bielefeld, 33615 Bielefeld, Germany, \\
Max Planck Institute for Mathematics in the Sciences, 04103 Leipzig, Germany}
\author[W. Liu]{Wei Liu}
\email{weiliu@jsnu.edu.cn}
\address{School of Mathematics and Statistics, Jiangsu Normal
University,  221116 Xuzhou,  China }
\author[A. Schenke]{Andre Schenke}
\email{aschenke@uni-bielefeld.de}
\address{Fakult\"at f\"ur Mathematik, Universit\"at Bielefeld, 33615 Bielefeld, Germany}


\makeatother

  \providecommand{\definitionname}{Definition}
  \providecommand{\examplename}{Example}
  \providecommand{\propositionname}{Proposition}
  \providecommand{\theoremname}{Theorem}
  \providecommand{\remarkname}{Remark}

\begin{document}

\begin{abstract}
We prove the existence of random dynamical systems and random attractors for a large class of locally monotone stochastic partial differential equations perturbed
by additive L\'{e}vy noise. The main result is applicable to various types of SPDE such as stochastic Burgers type equations, stochastic 2D Navier-Stokes
equations, the stochastic 3D Leray-$\alpha$ model, stochastic power law fluids, the stochastic Ladyzhenskaya model, stochastic Cahn-Hilliard type equations, stochastic Kuramoto-Sivashinsky type equations, stochastic porous media equations and stochastic $p$-Laplace equations.

\end{abstract}
\maketitle

\section{Introduction}
%

Since the foundational work in \cite{CF94,CDF97,S92} the long time behavior of SPDE in terms of the existence of random attractors has been extensively investigated (cf.\ e.g.\ \cite{BLL06, CSY15, Fan04, FGS17, Gess14, G13, G13-4, GLR11, BGLR10, GL17, GT16, GW18, KS04, LG08, LY16b, STW19, Wang17, YY14, Zhao17, Zhao17b, ZZ17}), resulting in an ever increasing list of specific SPDE for which the existence of a random attractor has been verified. While the proofs rely on common ideas, the field yet lacks a general, unifying framework overcoming the case by case verification. The main aim of this work is to further push in the direction of such a unifying framework by providing a general, abstract result on the existence of random attractors for locally monotone SPDE. 

More precisely, we prove the existence of random dynamical systems and random attractors for SPDE of the form
\begin{equation}
dX_{t}=A(X_{t})dt+dN_{t},\label{eqn:SPDE}
\end{equation}
where $N_{t}$ is a L\'{e}vy type noise satisfying a moment condition $(N)$ and $A$ is locally monotone (cf.\ $(A2)$ below) with respect to a Gelfand triple $V\subseteq H \subseteq V^*$. The abstract framework introduced here relies on the concept of \textit{locally} monotone operators. This extends previously available results, which were restricted to monotone operators, and constitutes important progress in so far that, in contrast to the monotone framework, it includes SPDE arising in fluid dynamics as particular examples. Indeed, the generality of this framework is demonstrated by application to a large class of SPDE, including, stochastic reaction-diffusion equations, stochastic Burgers type equations, stochastic 2D Navier-Stokes equations, the stochastic Leray-$\alpha$ model, stochastic power law fluids, the stochastic Ladyzhenskaya model, stochastic Cahn-Hilliard type equations as well as stochastic Kuramoto-Sivashinsky type equations. This recovers results from the literature as simple applications of the abstract framework introduced here and generalizes many known results. In particular, we generalize the results given in \cite{BGLR10,GLR11,G13}. We refer to Section 6 for more details.

The first main result, stated in detail in Theorem \ref{thm:generation} below, addresses the existence of random dynamical systems associated to \eqref{eqn:SPDE}.
\begin{thm*}[Theorem \ref{thm:generation} below]
Assume that $A$ is hemicontinuous, locally monotone, coercive and satisfies a growth condition. Further assume that $V \subseteq H$ is compact and that there exists a hemicontinuous, strictly monotone operator $M:V\to V^*$ satisfying a growth condition. Then there is a continuous random dynamical system $S$ generated by solutions to \eqref{eqn:SPDE}.
\end{thm*}

Under a slightly stronger coercivity condition we then prove the existence of a random attractor, leading to the second main result.
\begin{thm*}[Theorem \ref{thm:superlinear_ra} below]
Assume that $A$ is hemicontinuous, locally monotone, coercive and satisfies a growth condition. Further assume that $V \subseteq H$ is compact and that there exists a hemicontinuous, strictly monotone operator $M:V\to V^*$ satisfying a growth condition. Then the random dynamical system $S$ is compact and there is a random attractor for $S$.
\end{thm*}

The existence of a random attractor is typically proven in two steps: In the first step, uniform bounds on the $H$-norm of the flow are established, which means that there exists a bounded attracting set. In the second step, the existence of a compact attracting set is shown. In this work, we will use the compactness of the embedding $V \subseteq H$ to prove that the cocycle $S$ is compact, which together with the first step implies the existence of a compact attracting set. Notably, the approach introduced here only relies on the standard coercivity assumption of the variational approach to SPDE. This avoids further assumptions typically required in the literature in order to prove higher regularity of solutions. In particular, this avoids to pose stronger regularity assumptions on the noise.

The generality of the framework of locally monotone SPDE \eqref{eqn:SPDE} driven by additive L\'{e}vy noise results in several technical difficulties: The inclusion of additive, trace-class L\'{e}vy noise requires more involved estimates, e.g.\ in the proof of exponential integrability properties of the strictly stationary Ornstein-Uhlenbeck process. In addition, the more general growth assumptions introduced by the second author and R\"ockner in \cite{LR10} (cf.\ $(A4)$ below), lead to difficulties in controlling the coercivity and growth properties of the transformed (random) PDE, thus requiring more involved estimates than in previous works. Another key contribution of the present work is the detailed treatment of a rather long list of examples, which rely on an intensive use of interpolation inequalities, and which underlines the generality of the abstract framework. 

Notably, the L\'{e}vy process $N_{t}$ in \eqref{eqn:SPDE} is only assumed to take values in $H$ which is the natural choice of noise as far as trace-class noise is considered. This is in contrast to a number of works where the noise was assumed to take values in the domain of the operator $A$, in order to make sense of the transformed equation for $\tilde{Z}_{t} := X_{t} - N_{t}$ which has the form
$$
	d\tilde{Z}_{t} = A(\tilde{Z}_{t} + N_{t})dt.
$$
It was later noticed in \cite{G13} that this assumption can be relaxed to $N_{t} \in H$ by not subtracting the noise directly, but a form of \textit{nonlinear} Ornstein-Uhlenbeck process instead. More precisely, if the operator $A$ possesses a strongly monotone part $M$, we construct in Theorem \ref{thm:nonlinear_OU} a strictly stationary solution  $u_{t}$ of the equation (for sufficiently large $\sigma > 0$)
$$
	du_{t} = \sigma M(u_{t})dt + dN_{t}.
$$
Here, the smoothing properties of the operator $M$ guarantee that $u_{t}$ takes values in the space $V$. This allows to prove the existence of a random dynamical system, assuming only trace-class noise in $H$.

\subsection*{Literature}

We now give a brief account on the available literature on random attractors for SPDE. Since this is a very active research field, this attempt has to remain incomplete and we restrict to those works which appear most relevant to the results of this work.

Random attractor was first studied in  \cite{CF94,CDF97,S92}. It is a very important concept of capturing the long-time behavior of random dynamical systems (RDS) and there are many results on existence and properties of random attractors for various SPDE \cite{CSY15, Fan04, FGS17, Gess14, GL17, GT16, GW18, KS04, LG08, LY16b, STW19, Wang17, YY14, Zhao17, Zhao17b, ZZ17}.

Equivalent conditions for the existence of random attractors were given in \cite{CDS09}. Further properties of random attractors that have been studied include measurability \cite{CF94, CDF97, CLL18}, upper-semicontinuity \cite{CLR98, CL03, LCL14, Wang14, YL17}, regularity \cite{GLWY18, LGL15, LY16} and dimension estimates \cite{Debussche97, Langa03, ZZ16}. The problem of unbounded domains has also been addressed, e.g. in \cite{BL06, BLW09, HQWZ19, KW14, LWW18, Wang09}. For random attractors on weighted spaces, cf. \cite{BLW13, BLW14, LLL18}. Furthermore, the concept of a weak random attractor has been introduced recently in \cite{Wang18, Wang19}. Further references will be given in the discussion of the examples in Section \ref{sec:appl}.

Stochastic (partial) differential equations driven by L\'{e}vy noise have been studied widely, motivated among other things by applications in finance, statistical mechanics, fluid dynamics. For an overview we refer to \cite{PZ07}. For results on random attractors see \cite{GLR11} and the references therein. Well-posedness for locally monotone SPDE driven by L\'{e}vy noise was first studied by Brze\'{z}niak, the second author and Zhu \cite{BLZ12}.

\subsection*{Overview}

In Section \ref{sec:setup} we will state the assumptions on the coefficients and the noise $N$. In Section \ref{sec:strict_stat_sol} we will study strictly stationary solutions for strongly monotone SPDE. The following section, Section \ref{sec:construction} is devoted to constructing a stochastic flow via transformation of equation \eqref{eqn:spde2} into a random PDE. This stochastic flow is then proven to be compact in Section \ref{sec:existence_ra}. Combining with the existence of a random bounded absorbing set then immediately imply the existence of a random attractor. Applications to various SPDE are given in Section \ref{sec:appl}. Appendix \ref{app:var_spde} gathers the necessary results on random PDE with locally monotone coefficients. In Appendix \ref{app:rds} we will recall the basic notions and results concerning stochastic flows, random dynamical systems and random attractors.

\section{Main framework}\label{sec:setup}

Let $(H,\<\cdot,\cdot\>_{H})$ be a real separable Hilbert space, identified with its dual space $H^{*}$ by the Riesz isomorphism.
Let $V$ be a real reflexive Banach space continuously and densely embedded into $H$.
In particular, there is a constant $\l>0$ such that $ \l \|v\|_{H}^{2}\le \|v\|_{V}^{2}$ for all $v\in V$. Then we have the following Gelfand triple
\[
V\subseteq H\equiv H^{*}\subseteq V^{*}.
\]
If $_{V^{*}}\<\cdot,\cdot\>_{V}$ denotes the dualization between $V$ and its dual space $V^{*}$, then
\[
_{V^{*}}\<u,v\>_{V}=\<u,v\>_{H},\quad\forall u\in H,v\in V.
\]

As mentioned in the introduction, we consider SPDE of the form
\begin{equation}
dX_{t}=A(X_{t})dt+dN_{t},\label{eqn:spde2}
\end{equation}
where $A:V\to V^{*}$ is $\mcB(V)/\mcB(V^{*})$-measurable (we extend $A$ by $0$ to $H$) and $N:\R\times\O\to H$ is a centered,
two-sided L\'{e}vy process on $H$.
We assume that $N$ is given by its canonical realization on $\O:=D(\R;H)$,
the space of all \cadlag\  paths in $H$ endowed with the canonical filtration
\[
\mcF_{s}^{t}=\s(\o(u)-\o(v)|\o\in\O,\ s\le u,v\le t)
\]
and Wiener shifts $\{\t_{t}\}_{t\in\R}$ (cf. e.g. \cite[Appendix A.3]{A98}, \cite[Section 1.4.1]{A09}).
We will impose some moment condition on $N$ which will be specified below.
Let $\P$ be the law of $N$ on $\O$. Then $(\O,\mcF,\{\mcF_{s}^{t}\}_{t\in [s , \infty)},\{\t_{t}\}_{t\in\R},\P)$ is an ergodic metric dynamical system.
We denote the augmented filtration by $\{\bar{\mcF}_{s}^{t}\}_{t\in [s , \infty)}$ and note that $\{\bar{\mcF}_{s}^{t}\}_{t\in [s , \infty)}$ is right-continuous. The extension of $\P$ to $\bar{\mcF}$ is denoted by $\bar{\P}$ and we define $\bar{\mcF}_{-\infty}^{t} := \sigma \left( \bigcup_{- \infty < s \leq t} \bar{\mcF}_{s}^{t} \right)$.

Suppose that for some $\alpha\ge2$ and $\beta\ge0$ with $\b(\a-1)\le2$, there exist constants $C, K\ge0$ and $\gamma>0$
 such that the following conditions hold for all $v,v_{1},v_{2}\in V$:
\begin{enumerate}
\item [$(A1)$](Hemicontinuity) The map $s\mapsto{}_{V^{*}}\<A(v_{1}+sv_{2}),v\>_{V}$ is continuous on $\mathbb{R}$.
\item [$(A2)$](Local monotonicity)
\[
2{}_{V^{*}}\<A(v_{1})-A(v_{2}),v_{1}-v_{2}\>_{V}\le\left(C+\eta(v_{1})+\rho(v_{2})\right)\|v_{1}-v_{2}\|_{H}^{2},
\]
where $\eta,\rho:V\rightarrow\R_{+}$ are locally bounded measurable functions.
\item [$(A3)$](Coercivity)
\[
2{}_{V^{*}}\<A(v),v\>_{V}\le-\gamma\|v\|_{V}^{\alpha}+K\|v\|_{H}^{2}+C.
\]

\item [$(A4)$](Growth) 

\[
\|A(v)\|_{V^{*}}^{\frac{\a}{\a-1}}\le C(1+\|v\|_{V}^{\alpha})(1+\|v\|_{H}^{\b}).
\]

\end{enumerate}

In order to be able to deduce the existence and uniqueness of solutions from the results derived in \cite{BLZ12},
we note that due to the L\'{e}vy-It\^{o} decomposition (cf. e.g. \cite[Theorem 4.1]{AR05}),
and since $\nu$ is assumed to have first moment, we have $\P$-a.s.
\begin{equation}\label{eq:LevyIto}
N_{t}=m t+W_{t}+\int_{H}z\td N(t,dz),\quad\forall t\in\R,
\end{equation}
where $m\in H$, $W_{t}$ is a trace-class $Q$-Wiener process on $H$ and $\td N$ is a compensated Poisson random measure
on $H$ with intensity measure $\nu$ (cf. \cite{AR05} for Definitions).
Now we state the assumptions on the L\'{e}vy noise as follows:

$(N)$ The process $(N_{t})_{t \in \mathbb{R}}$ is a two-sided L\'{e}vy process with values in $H$ and the corresponding L\'{e}vy measure has finite moments up to order $4$. Furthermore, without loss of generality, we assume $m=0$.

Throughout this paper we will work with the convention that $C,\td C\ge0$ and $c,\td c>0$ are generic constants,
each of which is not important for its specific value and allowed to change from line to line.

Let us now define what we mean by a solution to \eqref{eqn:spde2}.
\begin{defn}
\label{def:soln_pathw} A \cadlag, $H$-valued, $\{\bar{\mcF}_{s}^{t}\}_{t \in [s,\infty)}$-adapted process $\{S(t,s;\o)x\}_{t\in[s,\infty)}$ is a solution to
\eqref{eqn:spde2} with initial condition $x$ at time $s$ if for $\P$-a.a.\ $\o\in\O$, $S(\cdot,s;\o)x\in L_{loc}^{\a}([s,\infty);V)$ and
\[
S(t,s;\o)x=x+\int_{s}^{t}A(S(r,s;\o)x)dr+N_{t}(\o)-N_{s}(\o),\quad\forall t\ge s.
\]

\end{defn}

\section{Strictly stationary solutions for monotone SPDE\label{sec:strict_stat_sol}}

The construction of stochastic flows for locally monotone SPDE driven by L\'{e}vy noise (presented in Section \ref{sec:construction} below)
will be based on strictly stationary solutions for strongly monotone SPDE driven by L\'{e}vy noise.
The existence and uniqueness of such strictly stationary solutions will be proven in this section, which might be of independent interest.
This generalizes a similar construction presented in \cite{G13} for the case of trace-class Wiener noise.

More precisely, in this section we will consider strongly monotone SPDE of the form
\begin{equation}
dX_{t}=\sigma M(X_{t})dt+dN_{t},\label{eqn:stricly_mon_SPDE}
\end{equation}
where $\sigma>0$, $N_{t}$ is a two-sided L\'{e}vy process (as above) and $M:V\to V^{*}$ is measurable.
Instead of the local monotonicity condition $(A2)$, we assume that $M$ is strongly monotone, i.e.
\begin{enumerate}
\item [($A2^\prime$)](Strong monotonicity) There exists a constant $c>0$ such that
\[
2{}_{V^{*}}\<M(v_{1})-M(v_{2}),v_{1}-v_{2}\>_{V}\le-c\|v_{1}-v_{2}\|_{V}^{\a},\quad\forall v_{1},v_{2}\in V,
\]
where $\a$ is the same constant as in $(A3)$.
\end{enumerate}
It is easy to see that $(A2')$ implies that $(A3)$ also holds for $M$.

By the above L\'{e}vy-It\^{o} decomposition \eqref{eq:LevyIto}, we may rewrite \eqref{eqn:stricly_mon_SPDE} as
\begin{equation}
dX_{t}= \sigma M(X_{t})  dt + dW_{t}+\int_{H}z\td N(d t, dz)  \label{eqn:SPDE_2}
\end{equation}
and \cite[Theorem 1.2]{BLZ12} implies the existence and uniqueness of an $\bar{\mcF}_{s}^{t}$-adapted variational solution $X(t,s;\o)x$ for each $x\in
H$.
Strictly stationary solutions to \eqref{eqn:stricly_mon_SPDE} will be constructed by letting $s\to-\infty$ in $X(t,s;\o)x$ and then selecting a strictly
stationary version $u$ from the resulting stationary
limit process using Proposition \ref{prop:stat_perfect} in Appendix B.

\begin{thm}[Strictly stationary solutions]
\label{thm:nonlinear_OU} Suppose that $M$ satisfies $(A1)$, $(A2')$,  $(A4)$ with $\beta=0$ and let $X(\cdot,s;\o)x$ be the solution to
\eqref{eqn:SPDE_2} starting in $x\in H$ at time $s$. Then
\begin{enumerate}
\item [ $(i)$ ] There exists an $\bar{\mcF}_{-\infty}^{t}$-adapted, $\mcF$-measurable process $u\in L^{2}(\O;D(\R;H))\cap
    L^{\a}(\O;L_{loc}^{\a}(\R;V))$ such that
\[
\lim_{s\to-\infty}X(t,s;\cdot)x=u_{t}
\]
in $L^{2}(\O;H)$ for all $t\in\R$, $x\in H$.
\item [$(ii)$]  $u$ solves \eqref{eqn:stricly_mon_SPDE} in the following sense:
\begin{equation}
u_{t}=u_{s}+\sigma\int_{s}^{t}M(u_{r})dr+N_{t}-N_{s},\ \P\text{-a.s., }\ t\ge s.\label{eqn:all_time_spde}
\end{equation}
\item [$(iii)$] $u$ can be chosen to be strictly stationary with \cadlag\  paths and satisfying $u_{\cdot}(\o)\in L_{loc}^{\a}(\R;V)$, for all
    $\o\in\O$.
\item [$(iv)$] Let $2 \le p \le 4$, then for each $\delta \ge 0$, $t\in\R$  and large enough $\sigma > \frac{8 \delta}{p c \lambda}$, there is a constant $C(\delta,\sigma)>0$ such that
\begin{equation}
\E\int_{-\infty}^{t}e^{\delta r}\|u_{r}\|_{V}^{\a}\|u_{r}\|_{H}^{p-2}dr\le C(\delta,\sigma)e^{\delta t},\label{eqn:u_V_bound}
\end{equation}
where $C(\delta,\sigma)\to0$ for $\sigma\to\infty$.
\item [$(v)$] There exists a $\t$-invariant set $\O_{0}\subseteq\O$ of full $\P$-measure such that for $\o\in\O_{0}$ and $s,t \in \R$, $s < t$,
\[
\frac{1}{t-s}\int_{s}^{t}\|u_{r}(\o)\|_{V}^{\a}dr\to\E\|u_{0}\|_{V}^{\a}\le C(\sigma),  \  s\to-\infty,
\]
where $C(\sigma)\to0$ for $\sigma\to\infty$.
\end{enumerate}
Let $p\in\N,$ $p\ge2$, then
\begin{enumerate}
\item [ $(vi)$] There exists a $\t$-invariant set $\O_{0}\subseteq\O$ of full $\P$-measure such that for $\o\in\O_{0}$
\begin{equation}
\frac{1}{t}\int_{0}^{t}\|u_{r}(\o)\|_{H}^{p}dr\to\E\|u_{0}\|_{H}^{p},\quad t\to\pm\infty.\label{eqn:OU_bound}
\end{equation}
\item[$(vii)$]$\|u_{t}(\o)\|_{H}^{p}$ has sublinear growth, i.e.
\[
\lim_{t\to\pm\infty}\frac{\|u_{t}(\o)\|_{H}^{p}}{|t|}=0.
\]
\end{enumerate}
\end{thm}

\begin{proof}
As the operator in \eqref{eqn:SPDE_2} is strongly monotone, some parts of the proof here are similar to  the associated statements in \cite{G13}. So here we  will only highlight the differences arising from allowing L\'{e}vy noise  and otherwise refer to \cite{G13}.

Let $X(t,s;\o)x$ denote the variational solution to \eqref{eqn:stricly_mon_SPDE} starting at time $s$ in $x\in H$ (cf. \cite{BLZ12}).

$(i)$ First we show that there is an $\bar{\mcF}_{-\infty}^{t}$-adapted, $\mcF$-measurable process $u:\R\times\O\to H$ such that
\[
\lim_{s\to-\infty}X(t,s;\cdot)x=u_{t},
\]
in $L^{2}(\O;H)$ for each $t\in\R$, independent of $x\in H$.

Following the same line of argument as in \cite[p. 143]{G13}, using the coercivity, It\^{o}'s formula and the comparison lemma \cite[Lemma 5.1]{G13} for $\alpha > 2$ or Gronwall's lemma for $\alpha = 2$, respectively, we obtain that for all $t\wedge0\ge s_{2}$
\[
\begin{split}\label{eqn:diff_{2}} & \E\sup_{r\in[t,\infty)}\|X(r,s_{2};\cdot)x-X(r,s_{1};\cdot)y\|_{H}^{2}\\
 & \hskip10pt\le\begin{cases}
\left((\frac{\a}{2}-1)c\sigma\lambda^{\frac{\a}{2}}(t-s_{2})\right)^{-\frac{2}{\a-2}} & \text{, if }\a>2;\\
2\left(e^{\frac{c\sigma\lambda}{2}s_{1}}\|y\|_{H}^{2}+e^{\frac{c\sigma\lambda}{2}s_{2}}\|x\|_{H}^{2}+C\right)e^{\frac{c\sigma\lambda}{2}s_{2}}e^{-c\sigma\lambda
t} & \text{, if }\a=2.
\end{cases}
\end{split}
\]
Hence, $X(\cdot,s;\cdot)x$ is a Cauchy sequence in $L^{2}(\O;D([t,\infty);H))$ and
\[
u_{t}:=\lim_{s\to-\infty}X(t,s;\cdot)x
\]
exists as a limit in $L^{2}(\O;H)$ for all $t\in\R$ and $u$ is $\bar{\mcF}_{-\infty}^{t}$-adapted.

Since $X(\cdot,s;\cdot)x$ also converges in $L^{2}(\O;D([t,\infty);H))$, $u$ is \cadlag\ $\P$-almost surely. Since $u$ is $\bar{\mcF}$-measurable, we can
choose an indistinguishable $\mcF$-measurable version of $u$.

$(ii)$  The next step consists of showing that  $u$ solves \eqref{eqn:all_time_spde}.

This is achieved using It\^{o}'s formula for $\| \cdot \|_{H}^{2}$ (with  the only difference being an additional term of $\int_{H} \| z \|_{H}^{2} \nu(dz)$ on the right-hand side), the compactness of the embedding as well as the monotonicity trick and the hemicontinuity $(A1)$. For details, cf. \cite[p.144]{G13}.

$(iii)$  Now we prove the crude stationarity for $u$. Let us first show $X(t,s;\o)x=X(0,s-t;\t_{t}\o)x$ for all $t\ge s$, $\P$-almost surely.

Let $h>0,t\ge s$ and define $\bar{X}_{h}(t)(\o):=X(t-h,s-h;\t_{h}\o)x$. Then for $\P$-a.a. $\o\in\O$ (with zero set possibly depending on $s,h,x$)
\begin{align*}
\bar{X}_{h}(t)(\o) & = X(t-h,s-h;\t_{h}\o)x\\
 & =x + \sigma \int_{s-h}^{t-h} M(X(r,s-h;\t_{h}\o)x)dr+N_{t-h}(\t_{h}\o)-N_{s-h}(\t_{h}\o)\\
 & =x + \sigma \int_{s-h}^{t-h}M(X(r,s-h;\t_{h}\o)x)dr+N_{t}(\o)-N_{s}(\o)\\
 & =x + \sigma\int_{s}^{t}M(\bar{X}_{h}(r)(\o))dr+N_{t}(\o)-N_{s}(\o).
\end{align*}
Hence, by uniqueness, $X(t-h,s-h;\t_{h}\o)x=X(t,s;\o)x$, $\P$-almost surely. In particular
\begin{equation}
X(0,s-t;\t_{t}\o)x=X(t,s;\o)x,\label{eqn:X-shift}
\end{equation}
$\P$-almost surely (with zero set possibly depending on $t,s,x$).

Now for an arbitrary sequence $s_{n} \to -\infty$ there exists a subsequence (again denoted by $s_{n}$) such that $X(t,s_{n};\cdot)x\to u_{t}$ and
$X(0,s_{n}-t;\cdot)x\to u_{0}$ $\P$-almost surely. Therefore, passing to the limit in \eqref{eqn:X-shift} gives
\[
u_{0}(\t_{t}\o)=u_{t}(\o),
\]
$\P$-almost surely (with zero set possibly depending on $t$).

Since $u_{\cdot}\in L^{\a}(\O;L_{loc}^{\a}(\R;V))$, hence in particular $u_{\cdot}(\o)\in L_{loc}^{\a}(\R;V)$ for almost all $\o\in\O$, and since $u$ is
$\mcF$-measurable,
we can employ Proposition \ref{prop:stat_perfect} to deduce the existence of an indistinguishable, $\mcF$-measurable,
$\bar{\mcF}_{-\infty}^{t}$-adapted, strictly stationary, \cadlag\  process $\tilde{u}$ such that $\tilde{u}_{\cdot}(\o)\in L_{loc}^{\a}(\R;V)$ for all
$\o\in\O$, i.e. crude stationarity.

$(iv)$  Next we proceed to prove \eqref{eqn:u_V_bound}. Let $\delta\ge0$ and note that by $(A2')$ and $(A4)$
\[
2\sigma\dualdel{V}{M(v)}{v}+\text{tr}Q \le-\frac{c\sigma}{2}\|v\|_{V}^{\a}+C, \quad\forall v\in V.
\]
An application of It\^{o}'s formula and the product rule yields that
\begin{align*}
 & e^{\delta t_{2}}\|u_{t_{2}}\|_{H}^{p}=e^{\delta t_{1}}\|u_{t_{1}}\|_{H}^{p}\\
 & \hskip15pt+\frac{p}{2}\int_{t_{1}}^{t_{2}}e^{\delta r}\|u_{r}\|_{H}^{p-2}\left(2\sigma\Vbk{M(u_{r}),u_{r}}+\text{tr}Q \right)dr\\
 & \hskip15pt+p\int_{t_{1}}^{t_{2}}e^{\delta r}\|u_{r}\|_{H}^{p-2}\<u_{r},dW_{r}\>_{H}+p(\frac{p}{2}-1)\int_{t_{1}}^{t_{2}}e^{\delta
 r}\|u_{r}\|_{H}^{p-4}\|Q^{\frac{1}{2}}u_{r}\|_{H}^{2}dr\\
 & \hskip15pt+p\int_{t_{1}}^{t_{2}}\int_{H}e^{\delta r}\|u_{r}\|_{H}^{p-2}\<u_{r},z\>_{H}\td N(dr,dz)\\
 & \hskip15pt+\int_{t_{1}}^{t_{2}}\int_{H}e^{\delta r} \left( \|u_{r}+z\|_{H}^{p}-\|u_{r}\|_{H}^{p}-p\|u_{r}\|_{H}^{p-2}\<u_{r},z\>_{H} \right)
 N(dr,dz)\\
 & \hskip15pt+\delta\int_{t_{1}}^{t_{2}}e^{\delta r}\|u_{r}\|_{H}^{p}dr.
\end{align*}
and thus by $(A3)$
\begin{align*}
 & \E e^{\delta t_{2}}\|u_{t_{2}}\|_{H}^{p}\le\E e^{\delta t_{1}}\|u_{t_{1}}\|_{H}^{p}\\
 & \hskip15pt+\frac{p}{2}\E\int_{t_{1}}^{t_{2}}e^{\delta r}\|u_{r}\|_{H}^{p-2}\left(-\frac{c\sigma}{2}\|u_{r}\|_{V}^{\a}+C\right)dr\\
 & \hskip15pt+p(\frac{p}{2}-1)\E\int_{t_{1}}^{t_{2}}e^{\delta r}\|u_{r}\|_{H}^{p-4}\|Q^{\frac{1}{2}}u_{r}\|_{H}^{2}dr\\
 & \hskip15pt+\E\int_{t_{1}}^{t_{2}}\int_{H}e^{\delta r} \left( \|u_{r}+z\|_{H}^{p}-\|u_{r}\|_{H}^{p}-p\|u_{r}\|_{H}^{p-2}\<u_{r},z\>_{H} \right)
 N(dr,dz)\\
 & \hskip15pt+\delta\E\int_{t_{1}}^{t_{2}}e^{\delta r}\|u_{r}\|_{H}^{p}dr.
\end{align*}
Noting that
\[
|\|x+h\|_{H}^{p}-\|x\|_{H}^{p}-p\|x\|_{H}^{p-2}\<x,h\>_{H}|\le C_{p}(\|x\|_{H}^{p-2}\|h\|_{H}^{2}+\|h\|^{p}),
\]
we obtain by using the moment assumption $(N)$ that

\begin{align*}
 & \E\int_{t_{1}}^{t_{2}}\int_{H}e^{\delta r} \left( \|u_{r}+z\|_{H}^{p}-\|u_{r}\|_{H}^{p}-p\|u_{r}\|_{H}^{p-2}\<u_{r},z\>_{H} \right) N(dr,dz)\\
 & =\E\int_{t_{1}}^{t_{2}}\int_{H}e^{\delta r} \left( \|u_{r}+z\|_{H}^{p}-\|u_{r}\|_{H}^{p}-p\|u_{r}\|_{H}^{p-2}\<u_{r},z\>_{H} \right) \nu(dz)dr\\
 & \le C\E\int_{t_{1}}^{t_{2}}\int_{H}e^{\delta r} \left( \|u_{r}\|_{H}^{p-2}\|z\|_{H}^{2}+\|z\|_{H}^{p} \right) \nu(dz)dr\\
 & \le C\left(\E\int_{t_{1}}^{t_{2}}e^{\delta r}\left(\|u_{r}\|_{H}^{p-2}+1\right) d r \right) ,
\end{align*}
and therefore
\begin{align*}
\E e^{\delta t_{2}}\|u_{t_{2}}\|_{H}^{p}\le~ & \E e^{\delta t_{1}}\|u_{t_{1}}\|_{H}^{p}-\frac{pc\sigma}{4}\E\int_{t_{1}}^{t_{2}}e^{\delta
r}\|u_{r}\|_{H}^{p-2}\|u_{r}\|_{V}^{\a}dr\\
 & +\left(p(\frac{p}{2}-1) \text{tr}Q +C\right)\E \int_{t_{1}}^{t_{2}}e^{\delta r}\|u_{r}\|_{H}^{p-2}dr \\
& +\delta\E\int_{t_{1}}^{t_{2}}e^{\delta r}\|u_{r}\|_{H}^{p}dr +C\int_{t_{1}}^{t_{2}}e^{\delta r}dr.
\end{align*}
Applying Young's inequality and the embedding $V \subset H$, we get
\begin{align*}
e^{\delta t_{2}}\E\|u_{t_{2}}\|_{H}^{p} & \le e^{\delta t_{1}}\E\|u_{t_{1}}\|_{H}^{p} - \left( \frac{pc\sigma}{4} - 2 \delta \lambda^{-1} \right)
\E\int_{t_{1}}^{t_{2}}e^{\delta r}\|u_{r}\|_{H}^{p-2}\|u_{r}\|_{V}^{\a}dr \\
&\quad +C \int_{t_{1}}^{t_{2}}e^{\delta r}dr.
\end{align*}
Stationarity of $u_{t}$ implies
\begin{align}
\left( \frac{pc\sigma}{4} - 2 \delta \lambda^{-1} \right) \E\int_{t_{1}}^{t_{2}}e^{\delta r}\|u_{r}\|_{H}^{p-2}\|u_{r}\|_{V}^{\a}dr\le C \int_{t_{1}}^{t_{2}}e^{\delta
r}dr\label{eqn:u_V_bound_1}
\end{align}
and thus \eqref{eqn:u_V_bound} holds, provided $\sigma$ is sufficiently large that $\frac{pc\sigma}{4} - 2 \delta \lambda^{-1} > 0$.

$(v)$  Applying \eqref{eqn:u_V_bound_1} for $\delta=0$ and $p=2$ yields
\[
\E\int_{t_{1}}^{t_{2}}\|u_{r}\|_{V}^{\a}dr\le C (t_2- t_1) .
\]
Since $u_{t}$ is stationary, we have $\E\|u_{r}\|_{V}^{\a}=\E\|u_{0}\|_{V}^{\a}$. Hence,
\[
\E\|u_{0}\|_{V}^{\a}\le  C<\infty,
\]
and Birkhoff's ergodic theorem implies the claimed convergence.

$(vi)$  The convergence \eqref{eqn:OU_bound} follows exactly as in \cite[Proof of Theorem 3.3 (i), 146 f.]{G13} from the stationarity and Birkhoff's ergodic theorem as well as an application of It\^{o}'s formula and the \emph{a priori} bounds arising from \cite[Lemma 5.2]{G13} in the case $\alpha > 2$ and Gronwall's lemma for $\alpha = 2$, respectively.

$(vii)$ This is proven by invoking the dichotomy of linear growth (cf. \cite[Proposition 4.1.3]{A98}) in the same way as in \cite[Proof of Theorem 3.3 (ii), p. 147]{G13}.
\end{proof}

\section{Generation of random dynamical systems\label{sec:construction}}

In order to construct a stochastic flow associated to \eqref{eqn:spde2}, we aim to transform \eqref{eqn:spde2} into a random PDE.
However, since we only assume that $N_{t}$ takes values in $H$ we cannot directly subtract the noise.
Motivated by \cite{G13} we use the transformation based on a strongly stationary solution to the strictly monotone part of \eqref{eqn:spde2}.
More precisely, we impose
the following assumption:
\begin{enumerate}
\item [$(V)$]There exists an operator $M:V\to V^{*}$ satisfying $(A1), (A2')$ and $(A4)$ with $\b=0$.
\end{enumerate}

The motivation behind the assumption $(V)$ is that $M$ is the strongly monotone part of $A$ in \eqref{eqn:spde2}. For example, for many semilinear SPDE
such as stochastic reaction-diffusion equations, stochastic Burgers equations and stochastic 2D Navier-Stokes equations, one can take $M=\Delta$
(standard Laplace operator). For quasilinear SPDE like stochastic porous media equations, stochastic $p$-Laplace equations or stochastic Cahn-Hilliard type equations one can take
$M(v)=\Delta(|v|^{r-1}v)$, $M(v)=\div(|\nabla v|^{p-2}\nabla v)$ and $M(v) = -\Delta^{2} v$, respectively (see Section \ref{sec:appl} for more concrete examples).

Following the arguments given in \cite{G13}, for $\sigma>0$ we may  consider the $\mcF$-measurable, strictly stationary solution $u_{t}$ (given by
Theorem \ref{thm:nonlinear_OU}) to
\[
du_{t}=\sigma M(u_{t})dt+dN_{t}.
\]
The key point is that $u$ takes values in $V$, while $N$  takes values in $H$. The operator $M$ is used to construct Ornstein-Uhlenbeck type process corresponding to $dX_{t}=\sigma M(X_{t})dt+dN_{t}$. If $N_{t}$ takes values in $V$ (cf.\cite{GLR11}), then this regularizing property is not needed and we can just choose $M=-Id_{H}$. The condition $(V)$ can be removed in this case.

Let $X(t,s;\o)x$ denote a variational solution to \eqref{eqn:spde2} starting in $x$ at time $s$ (the existence and uniqueness of this solution
will be proved in Theorem  \ref{thm:generation}).

Defining $\bar{X}(t,s;\o)x:=X(t,s;\o)x-u_{t}(\o)$ we get
\begin{align*}
\<\bar{X}(t,s;\o)x,v\>_{H} & =\<x-u_{s},v\>_{H}+\int_{s}^{t}\ _{V^{*}}\<A(\bar{X}(r,s;\o)x+u_{r}),v\>_{V}dr\\
 & \hskip10pt-\sigma\int_{s}^{t}\ _{V^{*}}\<M(u_{r}),v\>_{V}dr,\ v\in V,\ \P\text{-}a.s.
\end{align*}
We have used the following stationary conjugation mapping
\begin{equation}
T(t,\o)y:=y-u_{t}(\o)\label{eqn:stat_conj}
\end{equation}
and the conjugated process $Z(t,s;\o)x:=T(t,\o)X(t,s;\o)T^{-1}(s,\o)x$ satisfies
\begin{equation}
Z(t,s;\o)x=x+\int_{s}^{t}\left(A(Z(r,s;\o)x+u_{r})-\sigma M(u_{r})\right)dr\label{eqn:transformed_spde_0}
\end{equation}
as an equation in $V^{*}$. Let
\[
A_{\o}(r,v):=\begin{cases}
A\left(v+u_{r}\right)-\sigma M(u_{r}),\  & \text{if }u_{r}\in V;\\
A\left(v\right),\  & \text{else,}
\end{cases}
\]
where for the simplicity of notations we suppressed the $\o$-dependency of $u_{r}$.

Since $u_{r}(\o)\in V$ for all $\o\in\O$ and a.a. $r\in\R$, from \eqref{eqn:transformed_spde_0} we obtain
\begin{align}
Z(t,s;\o)x=x+\int_{s}^{t}A_{\o}(r,Z(r,s;\o)x)\ dr.\label{eqn:transformed_spde}
\end{align}
In order to define the associated stochastic flow to \eqref{eqn:spde2}, we will first solve \eqref{eqn:transformed_spde} for each $\o\in\O$ and then set
\begin{equation}
S(t,s;\o)x:=T(t,\o)^{-1}Z(t,s;\o)T(s,\o)x.\label{eq:S-def}
\end{equation}
This will be done in detail in the proof of Theorem \ref{thm:generation} below. For this purpose and also in order to subsequently prove the compactness
of the stochastic flow, we need to impose the following additional assumption:
\begin{enumerate}
\item [(A5)] The embedding $V\subseteq H$ is compact. \end{enumerate}
\begin{thm}[Generation of stochastic flows]
\label{thm:generation} Suppose that $(A1)$--$(A5)$, $(V)$ are satisfied and  there exist non-negative constants $C$ and $\kappa$ such that
\begin{equation}
\begin{split}\eta(v_{1}+v_{2}) & \le C(\eta(v_{1})+\eta(v_{2})),\\
\rho(v_{1}+v_{2}) & \le C(\rho(v_{1})+\rho(v_{2})),\\
\eta(v)+\rho(v) & \le C(1+\|v\|_{V}^{\alpha})(1+\|v\|_{H}^{\kappa}),\quad\forall v_{1},v_{2},v\in V.
\end{split}
\label{conditon on eta and rho}
\end{equation}
Then we have the following:
\begin{enumerate}
\item [$(i)$] There is a unique solution $Z(t,s;\o)$ to \eqref{eqn:transformed_spde}. $Z(t,s;\o)$ and $S(t,s;\o)$ (defined in \eqref{eq:S-def}) are
    stationary conjugated continuous RDS in $H$ and $S(t,s;\o)x$ is a solution of \eqref{eqn:spde2} in the sense of Definition \ref{def:soln_pathw}.
\item [$(ii)$] The maps $t\mapsto Z(t,s;\o)x$, $S(t,s;\o)x$ are \cadlag, $x\mapsto Z(t,s;\o)x$, $S(t,s;\o)x$ are continuous locally uniformly in $s,t$
    and $s\mapsto Z(t,s;\o)x$, $S(t,s;\o)x$ are right-continuous. \end{enumerate}
\end{thm}

\begin{proof}
 $(i)$  We consider \eqref{eqn:transformed_spde} as an $\o$-wise random PDE. We will use this point of view to define the associated stochastic flow.

In order to obtain the existence and uniqueness of solutions to \eqref{eqn:transformed_spde} for each fixed $(\o,s)\in\O\times\R$, we need to verify the
assumptions $(H1)$--$(H4)$
(see  Appendix \ref{app:var_spde}) for $(t,v)\mapsto A_{\o}(t,v)$. We will check $(H1)$--$(H4)$ for $A_{\o}(t,v)$ on each bounded interval
$[S,T] \subset \R$ and for each fixed $\o \in \O$. For ease of notations we suppress the $\o$-dependency of the coefficients in the following calculations.

$(H1)$: Follows immediately from $(A1)$ for $A$.

$(H2)$: Let $v_{1},v_{2}\in V$, $\o\in\O$ and $t\in\R$ such that $u_{t}(\o)\in V$. Then by $(A2)$ and \eqref{conditon on eta and rho} we find
\begin{align*}
 & 2{}_{V^{*}}\<A_{\o}(t,v_{1})-A_{\o}(t,v_{2}),v_{1}-v_{2}\>_{V}\\
 & =2{}_{V^{*}}\<A\left(v_{1}+u_{t}\right)-A\left(v_{2}+u_{t}\right),(v_{1}+u_{t})-(v_{2}+u_{t})\>_{V}\\
 & \le\left(C+\eta(v_{1}+u_{t})+\rho(v_{2}+u_{t})\right)\|v_{1}-v_{2}\|_{H}^{2}\\
 & \le\left(C+C\eta(u_{t})+C\rho(u_{t})+C\eta(v_{1})+C\rho(v_{2})\right)\|v_{1}-v_{2}\|_{H}^{2}.
\end{align*}
Note that by \eqref{conditon on eta and rho}
\[
\eta(u_{t})+\rho(u_{t})\le C\left(1+\|u_{t}\|_{V}^{\alpha}\right)\left(1+\|u_{t}\|_{H}^{\kappa}\right).
\]
Since $u_{\cdot}(\o)\in L_{loc}^{\a}(\R;V)\cap L_{loc}^{\infty}(\R;H)$, we have
\[
f(t):=C+C\eta(u_{t})+C\rho(u_{t})\in L_{loc}^{1}(\R),
\]
i.e. $(H2)$ holds for $A_{\omega}$. For $t\in\R$ such that $u_{t}(\o)\not\in V$ a similar calculation holds.

$(H3)$: For $v\in V$, $\o\in\O$ and $t\in\R$ such that $u_{t}(\o)\in V$, by $(A3)$ we can estimate
\begin{align}
2{}_{V^{*}}\<A_{\o}(t,v),v\>_{V}=~ & 2 {}_{V^{*}}\<A\left(v+u_{t}\right),v+u_{t}\>_{V}\label{eqn:H3-1}\\
 & -2 {}_{V^{*}}\<A\left(v+u_{t}\right),u_{t}\>_{V}-2\sigma{}_{V^{*}}\<M(u_{t}),v\>_{V}\nonumber \\
\le~ & K \|v+u_{t}\|_{H}^{2}-\gamma \|v+u_{t}\|_{V}^{\a}+C\nonumber \\
 & +2 \|A\left(v+u_{t}\right)\|_{V^{*}}\|u_{t}\|_{V}-2 \sigma{}_{V^{*}}\<M(u_{t}),v\>_{V}.\nonumber
\end{align}
For any $\ve_{1},\ve_{2}>0$, by $(A4)$, the condition $(\alpha - 1)\beta \leq 2$ and Young's inequality there exist constants $C_{\ve_{1}},C_{\ve_{2}}$ such that
\begin{align*}
 & 2\|A\left(v+u_{t}\right)\|_{V^{*}}\|u_{t}\|_{V}\\
 & \le C\left(1+\|v+u_{t}\|_{V}^{\a-1}\right)\left(1+\|v+u_{t}\|_{H}^{\b\frac{\a-1}{\a}}\right)\|u_{t}\|_{V}\\
 & \le\ve_{1}\left(1+\|v+u_{t}\|_{V}^{\a}\right)+C_{\ve_{1}}\left(1+\|v+u_{t}\|_{H}^{\b(\a-1)}\right)\|u_{t}\|_{V}^{\a}\\
 & \le\ve_{1}\|v+u_{t}\|_{V}^{\a}+C_{\ve_{1}}\|u_{t}\|_{V}^{\a}\|v+u_{t}\|_{H}^{2}+2C_{\ve_{1}}\|u_{t}\|_{V}^{\a}+\ve_{1}
\end{align*}
and
\begin{align*}
2\sigma{}_{V^{*}}\<M(u_{t}),v\>_{V} & \le C_{\ve_{2}}\|M(u_{t})\|_{V^{*}}^{\frac{\a}{\a-1}}+\ve_{2}\|v\|_{V}^{\a}\\
 & \le C_{\ve_{2}}\Big(C\|u_{t}\|_{V}^{\a}+C\Big)+\ve_{2}\|v\|_{V}^{\a},
\end{align*}
where we recall that $M$ satisfies $(A4)$ with $\beta=0$.

Combining the above estimates with \eqref{eqn:H3-1} we have
\begin{align*}
2{}_{V^{*}}\<A_{\o}(t,v),v\>_{V} & \le(K+C_{\ve_{1}}\|u_{t}\|_{V}^{\a})\|v+u_{t}\|_{H}^{2}-(\gamma-\ve_{1})\|v+u_{t}\|_{V}^{\a}\\
 & \hskip10pt+2C_{\ve_{1}}\|u_{t}\|_{V}^{\a}+C+\varepsilon_{1}+C_{\ve_{2}}\big(C\|u_{t}\|_{V}^{\a}+C\big)+\ve_{2}\|v\|_{V}^{\a}.
\end{align*}
Using
\[
\|v+u_{t}\|_{V}^{\a}\ge2^{1-\a}\|v\|_{V}^{\a}-\|u_{t}\|_{V}^{\a}
\]
we obtain (for $\ve_{1}$ small enough):
\begin{align*}
 & 2{}_{V^{*}}\<A_{\o}(t,v),v\>_{V}\\
 & \le-(\gamma-\ve_{1}-2^{\alpha-1}\ve_{2})2^{1-\a}\|v\|_{V}^{\a} +2(K+C_{\ve_{1}}\|u_{t}\|_{V}^{\a})\|v\|_{H}^{2}  \\
 &
 \hskip10pt+(\gamma-\ve_{1}+2C_{\ve_{1}}+CC_{\ve_{2}})\|u_{t}\|_{V}^{\a}+2(K+C_{\ve_{1}}\|u_{t}\|_{V}^{\a})\|u_{t}\|_{H}^{2}+CC_{\ve_{2}}+C+\varepsilon_{1}.
\end{align*}
Now choosing $\ve_{1},\ve_{2}$ small enough yields
\begin{align}
2_{V^{*}}\<A_{\o}(t,v),v\>_{V} & \le-\tilde{\gamma}\|v\|_{V}^{\a}+g(t)\|v\|_{H}^{2}+\tilde{f}(t),\label{eqn:transf_coerc}
\end{align}
where
\begin{align*}
\tilde{\gamma} & :=(\gamma-\ve_{1}-2^{\alpha-1}\ve_{2})2^{1-\a}>0;\\
g(t) & :=2(K+C_{\ve_{1}}\|u_{t}\|_{V}^{\a})\in L_{loc}^{1}(\R);\\
\tilde{f}(t) & :=( \gamma-\ve_{1}+2C_{\ve_{1}}+CC_{\ve_{2}})\|u_{t}\|_{V}^{\a}\\
 & \hskip10pt+2(C+C_{\ve_{1}}\|u_{t}\|_{V}^{\a})\|u_{t}\|_{H}^{2}+CC_{\ve_{2}}+C+\varepsilon_{1}\in L_{loc}^{1}(\R).
\end{align*}
Here the local integrability of $g$ and $\tilde{f}$ follows from the local $L^{\alpha}$-integrability of $u$ in $V$ and local boundedness of $u$ in $H$.
For $t\in\R$ such that $u_{t}(\o)\not\in V$ we can use the same calculation to prove $(H3)$.

$(H4)$: For $v\in V$, $\o\in\O$ and $t\in\R$ such that $u_{t}(\o)\in V$:
\begin{align*}
\|A_{\o}(t,v)\|_{V^{*}}^{\frac{\a}{\a-1}} & =\|A\left(v+u_{t}\right)-\sigma M(u_{t})\|_{V^{*}}^{\frac{\a}{\a-1}}\\
 & \le C\left(\|A\left(v+u_{t}\right)\|_{V^{*}}^{\frac{\a}{\a-1}}+\|M(u_{t})\|_{V^{*}}^{\frac{\a}{\a-1}}\right)\\
 & \le C\left(1+\|v+u_{t}\|_{V}^{\a}\right)\left(1+\|v+u_{t}\|_{H}^{\b}\right)+C\|M(u_{t})\|_{V^{*}}^{\frac{\a}{\a-1}}\\
 & \le C\Big(1+\|v\|_{V}^{\a}+\|u_{t}\|_{V}^{\a}+\|v\|_{H}^{\b}+\|u_{t}\|_{H}^{\b}+\|v\|_{V}^{\a}\|v\|_{H}^{\b}\\
 & \hskip15pt+\|v\|_{V}^{\a}\|u_{t}\|_{H}^{\b}+\|u_{t}\|_{V}^{\a}\|v\|_{H}^{\b}+\|u_{t}\|_{V}^{\a}\|u_{t}\|_{H}^{\b}\Big)\\
 & \le C\left(1+\|u_{t}\|_{H}^{\b}\right)\|v\|_{V}^{\a}+C(1+\|u_{t}\|_{V}^{\a})\|v\|_{H}^{\b}+C\|v\|_{V}^{\a}\|v\|_{H}^{\b}\\
 & \hskip15pt+C\left(1+\|u_{t}\|_{V}^{\a}+\|u_{t}\|_{H}^{\b}+\|u_{t}\|_{V}^{\a}\|u_{t}\|_{H}^{\b}\right)\\
 & \le\left(C_{1}(t)+C_{2}(t)\|v\|_{V}^{\a}\right)\left(1+\|v\|_{H}^{\b}\right),
\end{align*}
where
\begin{align*}
C_{1}(t) & :=C\left(1+\|u_{t}\|_{V}^{\a}+\|u_{t}\|_{H}^{\b}+\|u_{t}\|_{V}^{\a}\|u_{t}\|_{H}^{\b}\right)\in L_{loc}^{1}(\R);\\
C_{2}(t) & :=C\left(1+\|u_{t}\|_{H}^{\b}\right)\in L_{loc}^{\infty}(\R).
\end{align*}
This yields $(H4)$ on any bounded interval $[S,T]\subseteq\R$. %

For $t\in\R$ such that $u_{t}\not\in V$ one can show $(H4)$ by a similar calculation.

Hence, $(H1)$--$(H4)$ are satisfied for $A_{\o}$ for each $\o\in\O$ and on each bounded interval $[S,T]\subseteq\R$. By Theorem \ref{thm:var_ex} there
thus exists a unique solution
\[
Z(\cdot,s;\o)x\in L_{loc}^{\a}([s,\infty);V)\cap L_{loc}^{\infty}([s,\infty);H)
\]
to \eqref{eqn:transformed_spde} for every $(s,\o,x)\in\R\times\O\times H$.

By the uniqueness of solutions for \eqref{eqn:transformed_spde} we have the flow property
\[
Z(t,s;\o)x=Z(t,r;\o)Z(r,s;\o)x.
\]
Therefore, by Proposition \ref{prop:def_conj_flow} the family of maps given by
\begin{equation}
S(t,s;\o):=T(t,\o)\circ Z(t,s;\o)\circ T^{-1}(s,\o)\label{eq:def_flow}
\end{equation}
defines a stochastic flow.

Strict stationarity of $u_{t}$ implies that $A_{\o}(t,v)=A_{\t_{t}\o}(0,v)$. By the uniqueness of solutions for \eqref{eqn:transformed_spde} we deduce
that
\begin{align*}
Z(t,s;\o)x=Z(t-s,0;\t_{s}\o)x
\end{align*}
and thus $Z(t,s;\o)x$ is a cocycle. Since $T(t,\o)$ is a stationary conjugation, the same holds for $S(t,s;\o)x$.

Measurability of $\o\mapsto Z(t,s;\o)x$  follows as in the proof of \cite[Theorem 1.4]{GLR11}. In fact, the same argument proves
$\bar{\mcF}_{s}^{t}$-adaptedness
of $\o\mapsto Z(t,s;\o)x$. Due to \eqref{eq:def_flow}, in order to deduce measurability and $\bar{\mcF}_{s}^{t}$-adaptedness of $S(t,s;\o)x$
we only need to prove local uniform continuity of $x\mapsto Z(t,s;\o)x$ which will be done in  $(ii)$ below. Then it is  simple  to show that
$S(t,s;\o)x$ is a solution to \eqref{eqn:spde2}.

$(ii)$  Since $T(t,\o)y$ is \cadlag\  in $t$ locally uniformly in $y$, $t\mapsto S(t,s;\o)x$ is \cadlag. Since $(H2)$ holds for $A_{\omega}$, by
Gronwall's lemma (cf. \cite[Theorem 5.2.4 (i), Eq. (5.32)]{LR15}) we have for $s\le t$,
\begin{align*}
&\|Z(t,s;\o)x-Z(t,s;\o)y\|_{H}^{2} \\
 &\quad \le  \exp\left[ \int_{s}^{t} \left(f(r)+\eta(Z(r,s;\omega)x)+\rho(Z(r,s;\omega)y) \right) dr \right]\|x-y\|_{H}^{2}.
\end{align*}

By (\ref{conditon on eta and rho}) for $y\in B(x,r):=\{y\in H|\ \|x-y\|_{H}\le r\}$ we have
\[
\int_{s}^{t}(f(r)+\eta(Z(r,s;\omega)x)+\rho(Z(r,s;\omega)y))dr\le C.
\]
Thus $x\mapsto Z(t,s;\o)x$ is continuous locally uniformly in $s,t$. Moreover,  for   $s_{1}<s_{2}$ we have
\begin{align*}
 & \|Z(t,s_{1};\o)x-Z(t,s_{2};\o)x\|_{H}^{2}\\
= & \|Z(t,s_{2};\o)Z(s_{2},s_{1};\o)x-Z(t,s_{2};\o)x\|_{H}^{2}\\
\le & \exp\left[  \int_{s_{2}}^{t} \left( f(r)+\eta(Z(r,s_{1};\omega)x)+\rho(Z(r,s_{2};\omega)x) \right) d r \right] \|Z(s_{2},s_{1};\o)x-x\|_{H}^{2},
\end{align*}
which implies right-continuity of $s\mapsto Z(t,s;\o)x$.

Right continuity of $s\mapsto S(t,s;\o)x$ and continuity of $x\mapsto S(t,s;\o)x$ locally uniformly in $s,t$ follow from the corresponding properties of $Z(t,s;\o)$.
\end{proof}

\section{Existence of a random attractor\label{sec:existence_ra}}

In the following let $\mcD$ be the system of all tempered sets. Now we are in a position to state the main result of this work.
\begin{thm}
\label{thm:superlinear_ra} Suppose that $(A1)$--$(A5)$, $(V)$ and $(\ref{conditon on eta and rho})$ hold and let $S(t,s;\o)$ be the continuous cocycle
constructed in Theorem \ref{thm:generation}. Then

(i) $S(t,s;\o)$ is a compact cocycle.

For $\a=2$ additionally assume $K<\frac{\gamma\lambda}{4}$ in $(A3)$. Then

(ii) there is a random $\mcD$-attractor $\mcA$ for $S(t,s;\o)$.
\end{thm}

As a first step of the proof of Theorem \ref{thm:superlinear_ra} we shall prove bounded absorption. Let $B(x,r):=\{y\in H|\ \|x-y\|_{H}\le r\}$. %

\begin{prop}[Bounded absorption]
\label{prop:bdd_abs} Assume $(A1)$--$(A5)$, $(V)$ and $(\ref{conditon on eta and rho})$. If $\a=2$, additionally assume $K<\frac{ \gamma\lambda}{4}$ in
$(A3)$. Then there is a random bounded $\mcD$-absorbing set $\{F(\o)\}_{\o\in\O}$ for $S(t,s;\o)$. \\
More precisely, there is a measurable function $R:\O\to\R_{+}\setminus\{0\}$ such that for all $D\in\mcD$ there is an absorption time $s_{0}=s_{0}(D;\o)$
such that
\begin{equation}
S(0,s;\o)D(\theta_{s}\o)\subseteq B(0,R(\o)),\quad\forall s\le s_{0},\ \P\text{-a.s.}\label{eqn:abs_by_ball}
\end{equation}
\end{prop}
\begin{proof}
By \eqref{eqn:transf_coerc} we have
\begin{align*}
2{}_{V^{*}}\<A_{\o}(t,v),v\>_{V} & \le -(\gamma-\ve_{1}-2^{\alpha-1}\ve_{2})2^{1-\a}\|v\|_{V}^{\a} +
2(K+C_{\ve_{1}}\|u_{t}\|_{V}^{\a})\|v\|_{H}^{2} +\td f(t).
\end{align*}
Note that for $\a=2$ we also have $K<\frac{\gamma\lambda}{4}$, and choosing $\ve_{1},\ve_{2}$ small enough, we conclude
\[
2\dualdel{V}{A_{\o}(t,v)}{v}\le  c(t,\o)\|v\|_{H}^{2}+\td f(t,\o),\quad\forall v\in V,
\]
where $c(t,\o):=-\td{c}+C\|u_{t}(\o)\|_{V}^{\a}$ and
\begin{align*}
\tilde{f}(t,\o)=C\left(1+\|u_{t}\|_{V}^{\a}+\|u_{t}\|_{H}^{2}+\|u_{t}\|_{H}^{2}\|u_{t}\|_{V}^{\a}\right)
\end{align*}
for some $C, \td {c}>0$.

Note that  $\td {c}$ does not depend on $\sigma$.
 For a.e.\ $t\ge s$ we obtain
\begin{align*}
\frac{d}{dt}\|Z(t,s;\o)x\|_{H}^{2} & =2\ _{V^{*}}\<A_{\o}(t,Z(t,s;\o)x),Z(t,s;\o)x\>_{V}\\
 & \le c(t, \o)\|Z(t,s;\o)x\|_{H}^{2}+\td f(t,\o).
\end{align*}
By Theorem \ref{thm:nonlinear_OU}, for sufficiently large $\sigma$, there is a subset $\O_{0}\subseteq\O$ of full $\P$-measure such that
\[
\frac{1}{-s}\int_{s}^{0}\|u_{\tau}(\o)\|_{V}^{\a}d\tau\to\E\|u_{0}\|_{V}^{\a}<\frac{\td {c}}{2C},\quad\text{for }s\to-\infty
\]
and $\|u_{t}(\o)\|_{H}^{2}\|u_{t}(\o)\|_{V}^{\a}$ is exponentially integrable for all $\o\in\O_{0}$.

Hence, there is an $s_{0}(\o)\le0$ such that
\[
\frac{1}{-s}\int_{s}^{0}\left(-\td {c}+C\|u_{\tau}\|_{V}^{\a}\right)\ d\tau\le- \frac{\td {c}}{2},
\]
for all $s\le s_{0}(\o)$, $\o\in\O_{0}$ and some $\td {c}>0$.

Let $D\in\mcD$, $x_{s}(\o)\in D(\theta_{s}\o)$.%
{} For some $\td s_{0}=\td s_{0}(D;\o)$, by Gronwall's lemma we obtain
\begin{equation}
\begin{split}  &  \|Z(0,s;\o) x_{s}(\o)\|_{H}^{2} \\
 \le & \|x_{s}(\o)\|_{H}^{2} e^{ \frac{\td c s }{2} } + \int_{s}^{s_{0}} \td e^{\frac{\td {c}}{2} } f(r,\o)dr  + \int_{s_{0}}^{0} e^{\int_{r}^{0} ( - \td{c} + C\| u_{r} \|_{V}^{\a}) d\tau } \td{f}(r,\o) dr \\
  \le & ~1 + \int_{-\infty}^{s_{0}} e^{ \frac{\td {c}}{2} r} \td f(r,\o) dr + \int_{s_{0}}^{0} e^{\int_{r}^{0} ( - \td{c} + C\| u_{r} \|_{V}^{\a}) d\tau } \td{f}(r,\o) dr \\
  =: & R(\o),\quad\forall s\le\td s_{0},\ \P-\text{a.s.,}
\end{split}
\label{eqn:bdd_abs}
\end{equation}
where the finiteness of the second term follows from the exponential integrability of $\td f$.

Since $T(t,\o)=T(\t_{t}\o)$ is a bounded tempered map, we find bounded absorption for $S(t,s;\o)$.
\end{proof}

\begin{proof}[Proof of Theorem \ref{thm:superlinear_ra}]
$(i)$  Compactness of the cocycles $S(t,s;\o)$, $Z(t,s;\o)$ follows as in \cite[Theorem 3.1]{G13-4}.

$(ii)$   We prove that $Z(t,s;\o)x$ is $\mcD$-asymptotically compact. By Proposition \ref{prop:bdd_abs} there is a random, bounded $\mcD$-absorbing set
$F$. Let
\[
K(\o):=\overline{Z(0,-1;\o)F(\theta_{-1}\o)},\ \forall\o\in\O.
\]
Since $F(\theta_{-1}\o)$ is a bounded set and $Z(t,s;\o)$ is a compact flow, $K(\o)$ is compact. Furthermore, $K(\o)$ is $\mcD$-absorbing:
\begin{align*}
Z(0,s;\o)D(\theta_{s}\o) & =Z(0,-1;\o)Z(-1,s;\o)D(\theta_{s}\o)\\
 & \subseteq Z(0,-1;\o)F(\theta_{-1}\o)\subseteq K(\o),
\end{align*}
for all $s\le s_{0}$ $\P$-almost surely. By Theorem \ref{thm:suff_cond_attr} this yields the existence of a random $\mcD$-attractor for $Z(t,s;\o)$ and
thus, by Theorem \ref{thm:conj_attractor} for $S(t,s;\o)$. %

\end{proof}

\section{Examples} \label{sec:appl}

%
The main results of Theorems \ref{thm:generation} and \ref{thm:superlinear_ra} are applicable to a large class of SPDE, which not only
generalizes/improves many existing results but also can be used to obtain the existence of random attractors for some new examples.
In this section, we mostly present those stochastic equations with a locally monotone operator in the drift,
hence the existing results of \cite{GLR11,G13-4,G13} concerning only monotone operators are not applicable to those examples. We gather the examples considered in these papers at the end of this section.

Here is an overview of the examples considered:
In Section \ref{sub:Burgers} we study general Burgers-type equations. Sections \ref{sub:Stochastic-2D-Navier-Stokes} and \ref{sub:Leray} are devoted to Newtonian fluids, in particular we study the 2D Navier-Stokes equations and the 3D Leray-$\alpha$ model. More similar examples where the framework can be applied are summarized in Remark \ref{rem:HD-eqns}. We then move on to non-Newtonian fluids in Sections \ref{sub:Power-Law-Fluids} and \ref{sub:Ladyzhenskaya}, where power law fluids and the Ladyzhenskaya model are discussed. Sections \ref{sub:CH} and \ref{sub:KS} are concerned with Cahn-Hilliard-type equations in the sense of\cite{NC90} and general Kuramoto-Sivashinsky-type equations. Finally, in Section \ref{sub:monotone} we show how the aforementioned equations with monotone operators can be embedded into framework presented here.

\textbf{Notations} In this section we use $D_{i}$ to denote the spatial derivative $\frac{\partial}{\partial x_{i}}$ and $\Lambda\subseteq\mathbb{R}^{d}$
is supposed to be an open, bounded domain with smooth boundary and outward pointing unit normal vector $n$ on $\partial\Lambda$. For the Sobolev space
$W_{0}^{1,p}(\Lambda,\R^{d})$ $(p\ge2)$ we always use the following (equivalent) Sobolev norm
\[
\|u\|_{1,p}:=\left(\int_{\Lambda}|\nabla u(x)|^{p}dx\right)^{\frac{1}{p}}.
\]
Most examples below will deal with equations for vector-valued quantities. However, in some examples like those of Sections \ref{sub:Burgers}, \ref{sub:CH} and \ref{sub:KS}, we are in the scalar-valued case. We use the same notation for $L^{p}$ and Sobolev spaces in either case, as there is no risk of confusion.
Thus, for $p\ge1$, let $L^{p}$ denote either the vector-valued $L^{p}$-space $L^{p}(\L,\R^{d})$ or the scalar-valued $L^{p}$-space $L^{p}(\L,\R)$, with norm $\|\cdot\|_{L^{p}}$.

For an $\R^{d}$-valued function
$u:\L\to\mathbb{R}^{d}$ we define

\[
u\cdot\nabla=\sum_{j=1}^{d}u_{j}\partial_{j}
\]
and for an $\R^{d\times d}$-valued function $M:\L\to\mathbb{R}^{d\times d}$
\[
\text{div}\left(M\right)=\left(\sum_{j=1}^{d}\partial_{j}M_{i,j}\right)_{i=1}^{d}.
\]

For the reader's convenience, we recall the following Gagliardo-Nirenberg interpolation inequality (cf. e.g.\cite[Theorem 2.1.5]{taira:95}).

If $m,n\in\mathbb{N}$ and $q\in[1,\infty]$ such that
\[
\frac{1}{q}=\frac{1}{2}+\frac{n}{d}-\frac{m \theta}{d},\ \frac{n}{m}\le\theta\le1,
\]
then there exists a constant $C>0$ such that
\begin{equation}
\|u\|_{W^{n,q}}\le C\|u\|_{W^{m,2}}^{\theta}\|u\|_{L^{2}}^{1-\theta},\ \ u\in W^{m,2}(\Lambda).\label{GN_inequality}
\end{equation}
In particular, if $d=2$, we have the following well-known estimate on $\mathbb{R}^{2}$ (cf. \cite{T01,LR10}):
\begin{equation}
\|u\|_{L^{4}}^{4}\le C\|u\|_{L^{2}}^{2}\|\nabla u\|_{L^{2}}^{2},\ u\in W_{0}^{1,2}(\Lambda).\label{2d}
\end{equation}

\subsection{Stochastic Burgers type and reaction diffusion equations}\label{sub:Burgers}

We consider the following semilinear stochastic equation
\begin{align}
\begin{split}\label{rde}dX_{t}=\left(\Delta X_{t}+\sum_{i=1}^{d}f_{i}(X_{t})D_{i}X_{t}+f_{0}(X_{t})\right)dt+dN_{t},\end{split}
\end{align}
for the scalar quantity $X$ on $\L$. Let $(N_{t})_{t \in \R}$ be an $L^{2}(\Lambda)$-valued two-sided L\'{e}vy process satisfying $(N)$. Suppose the coefficients satisfy the following conditions:
\begin{enumerate}
\item [(i)] $f_{i}$ is Lipschitz on $\mathbb{R}$ for all $i=1,\ldots,d$;
\item [(ii)] $f_{0}\in C^{0}(\R)$ satisfies
\begin{equation}
\begin{split}|f_{0}(x)| & \le C(\vert x\vert^{r}+1),\ x\in\mathbb{R};\\
(f_{0}(x)-f_{0}(y))(x-y) & \le C (1+|y|^{s})(x-y)^{2},\ x,y\in\mathbb{R}.
\end{split}
\label{c1}
\end{equation}
where $C,r,s$ are some positive constants. \end{enumerate}
\begin{example}
\label{thm:burgers_RDE}Assume

(1) If $d=1,r=2,s=2$,

(2) If $d=2,r=2,s=2$, and $f_{i},i=1,2,3$ are bounded,

(3) If $d=3,r=2,s=\frac{4}{3}$ and $f_{i},i=1,2,3$ are bounded measurable functions which are independent of $X_{t}$.

Furthermore assume that the constant $K$ in the condition $(A3)$ and the domain $\Lambda$ satisfy $K < \frac{\lambda}{8}$.

Then there is a continuous cocycle and a random attractor associated to $(\ref{rde})$.
\end{example}
\begin{proof} We consider the following Gelfand triple
\[
V:=W_{0}^{1,2}(\Lambda)\subseteq H:=L^{2}(\Lambda)\subseteq V^{*}=(W_{0}^{1,2}(\Lambda))^{*}
\]
and define the operator
\[
A(u) = \tilde{A}(u) + f_{0}(u) =\Delta u+\sum_{i=1}^{d}f_{i}(u)D_{i}u+f_{0}(u),\ u\in V.
\]
One can show that $A$ satisfies $(A1)$--$(A3)$ with $\alpha=2$ and $\gamma = \frac{1}{2}$ and a constant $K$ (see \cite[Example 3.2]{LR10}). For $(A4)$ we note that
\[
	\| \tilde{A}u + f_{0}(u) \|_{V^{*}}^{2} \leq C \left( \| \tilde{A}u  \|_{V^{*}}^{2} + \| f_{0}(u) \|_{V^{*}}^{2} \right).
\]
The first term satisfies
\[
	\| \tilde{A}u  \|_{V^{*}}^{2} \leq C (1 + \| u \|_{V}^{2} ) (1 + \| v \|_{H}^{\nu}),
\]
where $\nu = 2$ in case (1) and $\nu = 0$ in case (2). For the second term we note that by applying H\"older's inequality and \eqref{GN_inequality}
\begin{align*}
	|~_{V^{*}}\langle f_{0}(u), v \rangle_{V} |^{2} &\leq \begin{cases}
	\| v \|_{L^{\infty}}^{2} \left( 1 + \| u \|_{L^{2}}^{4} \right),	\quad  &d=1 \\
	\| v \|_{L^{2}}^{2} \left( 1 + \| u \|_{L^{4}}^{4} \right),	\quad  &d=2 \\
	\| v \|_{L^{6}}^{2} \left( 1 + \| u \|_{L^{12/5}}^{4} \right),	\quad  &d=3
\end{cases} \\
	&\leq \| v \|_{V}^{2} (1 + \| u \|_{V}^{2} \| u \|_{H}^{2}).
\end{align*}
Thus, $(A4)$ holds with $\alpha = \beta = 2$.

Note that $(A5)$, $(V)$ and (\ref{conditon on eta and rho}) hold obviously with $M=\D$, therefore, the assertion follows from Theorem
\ref{thm:generation} and Theorem \ref{thm:superlinear_ra}. \end{proof}

\begin{rem}
(1) If $d=1$, one may take $f_{1}(x)=x$ such that Theorem \ref{thm:burgers_RDE} can be applied to the classical stochastic Burgers equation
(i.e. $(\ref{rde})$ with $f_{0}\equiv0$). Note that we may also allow a polynomial perturbation $f_{0}$ in the drift of $(\ref{rde})$. Hence, Theorem \ref{thm:burgers_RDE} also covers
stochastic reaction-diffusion-type equations. Due to the restrictions of the variational approach to (S)PDE we can only consider reaction terms of at most quadratic growth. However, as outlined in \cite[Remark 4.6]{G13}, the main ideas apply to SRDE with higher-order reaction terms as well, e.g. using the mild approach to SPDE.

(2) The stochastic Burgers equation has been studied intensively over the last decades. E, Khanin, Mazel and  Sinai \cite{EKMS00} proved the existence of singleton random attractors in 1D for periodic boundary conditions and noise of spatial regularity $C^{3}$. Iturriaga and Khanin in \cite{IK03} generalized these periodic results to the multidimensional case with spatial $C^{4}$ noise. Bakhtin \cite{Bak07} studied the case on $[0,1]$ with random boundary conditions of Ornstein-Uhlenbeck-type. The case on the whole space driven by a space-time homogeneous Poisson point field was studied by Bakhtin, Cator and Khanin \cite{BCK14}.

In \cite{DPD07}, Da Prato and Debussche study the stochastic  Burgers equation on an interval with Dirichlet boundary conditions and for cylindrical Wiener noise. They note \cite[Remark 2.4]{DPD07} that one can prove existence of a random attractor using essentially the same techniques as \cite{CF94}. The theorems proved in the present paper extend the above results to the case of more general, rougher noise as well as to the more general class of equations of the form \eqref{rde}.
\end{rem}

\subsection{Stochastic 2D Navier-Stokes equation and other hydrodynamical models}\label{sub:Stochastic-2D-Navier-Stokes}

The next example is stochastic 2D Navier-Stokes equation driven by additive noise. The Navier-Stokes equation is an important model in
fluid mechanics to describe the time evolution of incompressible fluids. It can be formulated as follows
\begin{equation}
\begin{split}\partial_{t}u(t) & =\nu\Delta u(t)-(u(t)\cdot\nabla)u(t)+\nabla p(t)+f,\\
\text{div}(u) & =0,\ u|_{\partial\Lambda}=0,~u(0)=u_{0},
\end{split}
\label{3D NSE}
\end{equation}
where $u(t,x)=(u^{1}(t,x),u^{2}(t,x))$ represents the velocity field of the fluid, $\nu$ is the viscosity constant, $p(t,x)$ is the pressure and $f$ is a
(known) external force field acting on the fluid. The stochastic version was first considered by Bensoussan and Temam in \cite{BT73} and has since been studied intensively. Random attractors for additive (as well as linear multiplicative) Wiener noise were first obtained by Crauel and Flandoli \cite{CF94}.

 As usual we define 
 (cf. \cite[Theorems 1.4 and 1.6]{T01}):
\begin{equation}
\begin{split}H & =\left\{ u \in L^{2}(\Lambda; \mathbb{R}^{2}):\ \nabla\cdot u=0\ \text{in}\ \Lambda,\ u\cdot n=0\ \text{on}\ \partial\Lambda\right\} ;\\
V & =\left\{ u\in W_{0}^{1,2}(\Lambda; \mathbb{R}^{2}):\ \nabla\cdot u=0\ \text{in}\ \Lambda\right\} .
\end{split}
\label{eq:div_free_spaces}
\end{equation}
 The Helmholtz-Leray projection $P_{H}$ and the Stokes operator $L$ with viscosity constant $\nu$ are defined by
\[
P_{H}:L^{2}(\Lambda,\mathbb{R}^{2})\rightarrow H,\ \text{orthogonal projection};
\]
\[
L:H^{2,2}(\Lambda,\mathbb{R}^{2})\cap V\rightarrow H,\ Lu=\nu P_{H}\Delta u.
\]
We thus arrive at the following abstract formulation of the Navier-Stokes equation
\begin{equation}
u'=Lu+F(u)+f,\ u(0)=u_{0}\in H,\label{NSE}
\end{equation}
where $f\in H$ (for simplicity we write $f$ for $P_{H}f$ again) and
\begin{align*}
F:V\times V\rightarrow V^{\ast},\ F(u,v) :=-P_{H}\left[\left(u\cdot\nabla\right)v\right],\ F(u):= F(u,u).
\end{align*}
It is well known that $F:V\times V\to V^{\ast}$ is well-defined and continuous. Using the Gelfand triple $V\subset H\equiv H^{*}\subset V^{*}$, one sees
that $L$ extends by continuity to a map $L:V\to V^{*}$.
Now we consider a random forcing and thus obtain the stochastic 2D Navier-Stokes equation
\begin{equation}
dX_{t}=\left(LX_{t}+F(X_{t})+f\right)dt+dN_{t}, \label{SNSE}
\end{equation}
where $(N_{t})_{t \in \R}$ is a two-sided trace-class L\'{e}vy process in $H$ satisfying $(N)$.

\begin{example} (Stochastic 2D Navier-Stokes equation) There exists a continuous cocycle and a random attractor associated to $(\ref{SNSE})$.
\end{example}

\begin{proof} According to the result in \cite[Example 3.3]{LR10},
$(A1)$--$(A4)$ hold with $\alpha=\beta=2$, $\eta\equiv0$ and $\rho(v)=\|v\|_{L^{4}}^{4}$ and $K=0$. $(A5)$, $(V)$ and (\ref{conditon
on eta and rho}) hold obviously
(with $M:=L$). Therefore, the assertion follows from Theorem \ref{thm:generation} and Theorem \ref{thm:superlinear_ra}. \end{proof}

\begin{rem}\label{rem:HD-eqns}
(1) The above result improves the classical results in \cite[Theorem 7.4]{CF94}  and \cite[Example 3.1]{CDF97} by allowing more general types of noise.
Besides L\'{e}vy-type noise being allowed here, even for Wiener-type noise, we don't need impose
any further assumptions on the noise except those needed for the well-posedness of the equation.

(2) As we mentioned in the introduction, many other
hydrodynamical systems also satisfy the local monotonicity $(A2)$ and coercivity condition $(A3)$.
For example, Chueshov and Millet \cite{CM10} studied well-posedness and large deviation principles for abstract stochastic semilinear
equations (driven by Wiener noise), covering a wide class of fluid dynamical models.
In fact, they consider abstract equations of the form
$$
	du(t) = (Lu(t) + B(u(t),u(t)) + \mathcal{R}u(t)) dt + \sigma(u(t)) dW(t).
$$
The operator  $L$ is a linear unbounded, self-adjoint and negative definite operator with $V = D((-L)^{1/2})$, $H$ is a separable Hilbert space such that the Gelfand triple $V \subseteq H \subseteq V^{*}$ holds. In \cite{CM10}, the inclusions do not have to be compact, but we have to assume this. $\mathcal{R} \colon H \rightarrow H$ is a bounded linear operator. The bilinear operator $B$ satisfies certain continuity, symmetry and interpolation/growth conditions, cf.\ \cite[$(C1)$]{CM10}.
These assumptions imply the conditions of this article:

$(A1)$ is clear by the continuity assumptions on the operators. $(A2)$ has been shown in \cite[Eq. (2.8)]{CM10} for the operator $B$. For the other two operators this follows immediately. $(A3)$ with $\alpha = 2$, $\gamma = 1$ and $K = \| \mathcal{R} \|$ follows as by assumption $_{V^{*}}\langle B(v,v), v \rangle_{V} = 0$, and $(A4)$ with $\beta = 2$ is implied by
\begin{align*}
	| ~_{V^{*}}\langle B(v,v), u \rangle_{V} | = | ~_{V^{*}}\langle B(v,u), v \rangle_{V} | \leq C \|u \|_{V} \| v \|_{H} \| v \|_{V}.
\end{align*}
As we assumed bounded domains, $(A5)$ holds and finally $(V)$ holds for $M = L$. Since $\alpha = 2$, we get the additional constraint $K = \| \mathcal{R} \| < \frac{\lambda}{4}$.

Therefore, Theorem \ref{thm:generation} and Theorem \ref{thm:superlinear_ra} can be applied to show the existence of a continuous cocycle and of a random
attractor for all the hydrodynamical models studied in \cite{CM10} driven by additive L\'{e}vy-type noise.
These models include stochastic magneto-hydrodynamic equations, the stochastic Boussinesq model for the B\'{e}nard convection, the stochastic 2D
magnetic B\'{e}nard problem and the stochastic 3D Leray-$\alpha$ model driven by additive noise. For brevity we shall restrict our attention to one further example, namely the stochastic 3D Leray-$\alpha$ model. \end{rem}

\subsection{Stochastic 3D Leray-$\alpha$ model}\label{sub:Leray}

We now apply the main result to the 3D Leray-$\alpha$ model of turbulence, which is a regularization of the 3D Navier-Stokes equation and was first
considered by Leray \cite{L34} in order to prove the existence of a solution to the Navier-Stokes equation in $\mathbb{R}^{3}$. Here we use a special
smoothing kernel, which goes back to Cheskidov, Holm, Olson and Titi \cite{CH05} (cf.\cite{CTV07} for more references). It has been shown there that the 3D Leray-$\alpha$ model
compares successfully with experimental data from turbulent channel and pipe flows for a wide range of Reynolds numbers and therefore has the potential to
become a good sub-grid-scale large-eddy simulation model for turbulence. The (deterministic) Leray-$\alpha$ model can be formulated as follows:
\begin{equation}
\begin{split}\label{3D} & \partial_{t} u =\nu\Delta u-(v\cdot\nabla)u-\nabla p+f,\\
 & \div(u)=0,\ u_{|\partial\L}=0,\ u=v-\varepsilon^{2}\Delta v,
\end{split}
\end{equation}
where $\nu>0$ is the viscosity, $u$ is the velocity, $p$ is the pressure and $f$ is a given body-forcing term. Using the same divergence-free Hilbert
spaces $V$ and $H$ as in (\ref{eq:div_free_spaces}) (but in the 3D case), one can rewrite the stochastic Leray-$\alpha$ model in the following abstract form:
\begin{equation}
dX_{t}=\left(LX_{t}+F(X_{t},X_{t})+f \right)dt+dN_{t},\label{Leray}
\end{equation}
where $f\in H$, $(N_{t})_{t \in \R}$ is a trace-class L\'{e}vy process in $H$ satisfying condition $(N)$, and
\[
Lu=\nu P_{H}\Delta u,\ \ F(u,v)=-P_{H}\left[\left(\left(I-\varepsilon^{2}\Delta\right)^{-1}u\cdot\nabla\right)v\right].
\]
The stochastic 3D Leray-$\alpha$ model was studied by Deugoue and Sango in \cite{DS10} and Chueshov and Millet in \cite{CM10} for the case of  Brownian motion noise. The inviscid case $\nu = 0$ was investigated by Barbato, Bessaih and Ferrario in \cite{BBF14}. The model has also been extended to the case of 3D MHD equations by Deugou\'{e}, Razafimandimby and Sango in \cite{DRS12}.

\begin{example}\label{exa4}(Stochastic 3D Leray-$\alpha$ model)
There exists a continuous cocycle and a random attractor associated to $(\ref{Leray})$.
\end{example}
\begin{proof} Conditions $(A1)$--$(A4)$ have been checked above and in \cite[Example 3.6]{L11} with $\alpha = 2$, $K=0$, $\beta = 2$. Condition $(V)$ holds with $M:=L$ and $(A5)$ is clear.
The assertion now follows from Theorem \ref{thm:generation} and Theorem \ref{thm:superlinear_ra}.
\end{proof}

\begin{rem}
	To the best of our knowledge,  the existence of a random attractor seems to be new for this model.
\end{rem}

\subsection{Stochastic power law fluids}\label{sub:Power-Law-Fluids}

The next example is an SPDE model which describes the velocity field of a viscous and incompressible non-Newtonian fluid subject to random forcing
in dimension $2 \le d \le 4$. The deterministic model has been studied intensively in PDE theory (cf.\cite{FR,MNRR} and the references therein). For a vector
field $u:\Lambda\rightarrow\R^{d}$, we define the rate-of-strain tensor by
\[
e(u):\Lambda\rightarrow\R^{d}\otimes\R^{d};\ e_{i,j}(u)=\frac{\partial_{i}u_{j}+\partial_{j}u_{i}}{2},\ i,j=1,\ldots,d
\]
and we consider the case that the stress tensor has the following polynomial form:
\[
\tau(u):\Lambda\rightarrow\R^{d}\otimes\R^{d};\ \tau(u)=2\nu(1+|e(u)|)^{p-2}e(u),
\]
where $\nu>0$ is the kinematic viscosity and $p>1$ is a constant, and for $U \in \R^{d} \otimes \R^{d}$ we define $|U| = \left( \sum_{i,j = 1}^{d} |U_{ij}|^{2} \right)^{1/2}$

In the case of deterministic forcing, the dynamics of power law fluids can be modelled by the following PDE (cf.\cite[Chapter 5]{MNRR}):
\begin{equation}
\begin{split} & \partial_{t}u=\text{div}\left(\tau(u)\right)-(u\cdot\nabla)u-\nabla p+f,\\
 & \text{div}(u)=0,\ u|_{\partial\Lambda}=0,~u(0)=u_{0},
\end{split}
\label{eqn:power law fluids}
\end{equation}
where $u=u(t,x)=\left(u_{i}(t,x)\right)_{i=1}^{d}$ is the velocity field, $p$ is the pressure and $f$ is an external force.

\begin{rem} For $p=2$,  \eqref{eqn:power law fluids} describes Newtonian fluids and $(\ref{eqn:power law fluids})$ reduces to the classical Navier-Stokes equation $(\ref{3D NSE})$.

The cases $p \in (1,2)$ and $p \in (2,\infty)$ are called shear-thinning fluids and
 shear-thickening fluids, respectively. They have been widely studied in different fields of science and
engineering (cf.   e.g. \cite{FR,MNRR} and the references therein). \end{rem}

In this section, we will only consider the case $p \geq \frac{d+2}{2} \geq 2$, i.e. the shear-thickening case.

In the following we consider the Gelfand triple $V\subset H\subset V^{*},$ where
\begin{align*}
V & =\left\{ u\in W_{0}^{1,p}(\Lambda;\R^{d}):\ \nabla\cdot u=0\ \text{in}\ \Lambda\right\} ;\\
H & =\left\{ u\in L^{2}(\Lambda;\R^{d}):\ \nabla\cdot u=0\ \text{in}\ \Lambda,\ u\cdot n=0\ \text{on}\ \partial\Lambda\right\} .
\end{align*}
Let $P_{H}$ be the orthogonal (Helmholtz-Leray) projection from $L^{2}(\L,\R^{d})$ to $H$. As in Example
\ref{sub:Stochastic-2D-Navier-Stokes}, the operators
\[
\mathcal{N}:W^{2,p}(\Lambda;\R^{d})\cap V\rightarrow H,\ \mathcal{N}(u):=P_{H}\left[\text{div}(\tau(u))\right];
\]
\begin{align*}
&F: \left( W^{2,p}(\Lambda;\R^{d})\cap V \right) \times \left( W^{2,p}(\Lambda;\R^{d})\cap V \right) \rightarrow H;\ \\
&F(u,v):=-P_{H}\left[(u\cdot\nabla)v\right],\ F(u):=F(u,u)
\end{align*}
can be extended to the well defined operators:
\[
\mathcal{N}:V\rightarrow V^{*};\ F:V\times V\rightarrow V^{*}.
\]
In particular, one can show that
\[
_{V^{*}}\<\mathcal{N}(u),v\>_{V}=-\int_{\Lambda}\sum_{i,j=1}^{d}\tau_{i,j}(u)e_{i,j}(v)dx,\ u,v\in V;
\]
\[
_{V^{*}}\<F(u,v),w\>_{V}=-{}_{V^{*}}\<F(u,w),v\>_{V},\ {}_{V^{*}}\<F(u,v),v\>_{V}=0,\ u,v,w\in V.
\]
Now (\ref{eqn:power law fluids}) with random forcing can be reformulated in the following abstract form:
\begin{equation}
dX_{t}=(\mathcal{N}(X_{t})+F(X_{t})+f)dt+dN_{t},\label{PLF}
\end{equation}
with $f\in H$ and $N_{t}$ being a trace-class L\'{e}vy process in $H$ satisfying the condition $(N)$.

\begin{example}(Stochastic power law fluids) Suppose that $2 \leq d \leq 4$ and $p \in \left[ \frac{d+2}{2}, 3 \right]$, then
there exists a continuous cocycle and a random attractor associated to $(\ref{PLF})$. \end{example}

\begin{proof}
From \cite[Example 3.5]{LR12} we know that $(A1)$ and $(A2)$ hold with $\rho(v)=C_{\varepsilon}\|v\|_{V}^{\frac{2p}{2p-d}}$ and
$\eta\equiv0$, and the operator $M := \mathcal{N}$ is in fact strongly monotone. $(A3)$ holds with $\alpha=p$ and $K = 0$. Furthermore, we have
\[
\|F(v)\|_{V^{*}}\le\|v\|_{L^{\frac{2p}{p-1}}}^{2},\ v\in V.
\]
An application of the Gagliardo-Nirenberg interpolation inequality \eqref{GN_inequality} yields
\[
\|v\|_{L^{\frac{2p}{p-1}}} \le C\|v\|_{V}^{\theta}\|v\|_{H}^{1-\theta},
\]
with $\theta=\frac{d}{(d+2)p-2d}$. Note that $2 \theta \le p-1$ if $p\ge\frac{d+2}{2}$, hence the embedding $V \subseteq H$ implies
\begin{align*}
\| F(v) \|_{V^{*}} &\leq C \|v\|_{V}^{2 \theta}\|v\|_{H}^{2(1-\theta)} \leq C \| v \|_{V}^{2 \theta} \|v\|_{H}^{(p-1) - 2\theta} \|v\|_{H}^{2(1-\theta) - ( (p-1) - 2\theta )} \\
&\leq \| v \|_{V}^{p-1} \| v \|_{H}^{3 - p} \quad \Rightarrow \quad \| F(v) \|_{V^{*}}^{\frac{p}{p-1}} \leq C \| v \|_{V}^{p} \| v \|_{H}^{\frac{(3-p)p}{p-1}},
\end{align*}
which implies $\alpha = p$, $\beta = \frac{(3-p)p}{p-1}$. Since $p \le 3$, we get $\beta \ge 0$. The condition $\beta(\alpha-1) \le 2$ is equivalent to $(3-p)p \leq 2$, which is satisfied for $p \geq 2$.

It is also easy to see that
\[
\|\mathcal{N}(v)\|_{V^{*}}\le C(1+\|v\|_{V}^{p-1}),\ v\in V.
\]
Hence the growth condition $(A4)$ holds with the above $\alpha$ and $\beta$. $(V)$ and $(A5)$ are clearly satisfied. The assertion now follows from Theorem \ref{thm:generation} and Theorem \ref{thm:superlinear_ra} .
\end{proof}

\subsection{Stochastic Ladyzhenskaya model}\label{sub:Ladyzhenskaya}
The Ladyzhenskaya model is a higher order variant of the power law fluid where the stress tensor has the form
$$
	\tau(u) \colon \Lambda \rightarrow \R^{d} \otimes \R^{d}, \quad \tau(u) = 2 \mu_{0}(1 + |e(u)|^{2})^{\frac{p-2}{2}} e(u) - 2 \mu_{1} \Delta e(u) = \tau^{\mathcal{N}}(u) + \tau^{\mathcal{L}}(u).
$$
The model was pioneered by Ladyzhenskaya \cite{L70} and further analyzed by various authors (see \cite{DZ10} and the references therein). Compared to the power law fluids considered above, there is an additional fourth order term $\nabla \cdot (- 2 \mu_{1} \Delta e(u) )$ present in the equation.

The existence of random attractors for this model has been studied for $p \in (1,2)$, i.e. shear-thinning fluids, by Duan and Zhao in \cite{DZ10} and for $p > 2$ by Guo and Guo \cite{GG10}.

In this section we will apply the general framework to this model in the case $p \in (1,3]$, recovering the results of \cite{DZ10} and parts of the results of \cite{GG10}. This restriction on $p$ allows us to understand the nonlinear term as a perturbation of the linear term. It is necessary again due to the restriction $\beta (\alpha - 1) \leq 2$ which restricts the homogeneity in $(A4)$. Furthermore, applying the Gagliardo-Nirenberg inequality \eqref{GN_inequality}, we find a "maximal" range $(1,p_{c}] \subset (1,3]$ of parameters $p$ to which the method presented in this article applies.

In what follows, the exact form of the powers in the stress tensor does not play any role, i.e. the results apply just as well to the case
$$
\tilde{\tau}^{\mathcal{N}}(u) = 2 \mu_{0}(1 + |e(u)|)^{p-2} e(u).
$$
Consider the Gelfand triple $V \subset H \subset V^{*}$, where
\begin{align*}
V & =\left\{ u\in W_{0}^{2,2}(\Lambda;\R^{d}):\ \nabla\cdot u=0\ \text{in}\ \Lambda\right\} ;\\
H & =\left\{ u\in L^{2}(\Lambda;\R^{d}):\ \nabla\cdot u=0\ \text{in}\ \Lambda,\ u\cdot n=0\ \text{on}\ \partial\Lambda\right\} .
\end{align*}

Let $P_{H}$ be the orthogonal (Helmholtz-Leray) projection from $L^{2}(\L,\R^{d})$ to $H$. Similar to Examples
\ref{sub:Stochastic-2D-Navier-Stokes} and \ref{sub:Power-Law-Fluids}, the operators
\begin{align*}
& \mathcal{N}: C^{\infty}_{c}(\Lambda;\R^{d})\cap V\rightarrow H,\ \mathcal{N}(u):=P_{H}\left[\text{div}(\tau^{\mathcal{N}}(u))\right]; \\
&\mathcal{L}: C^{\infty}_{c}(\Lambda;\R^{d})\cap V\rightarrow H,\ \mathcal{L}u:=P_{H}\left[\text{div}(\tau^{\mathcal{L}}(u))\right]; \\
&F: \left( C^{\infty}_{c}(\Lambda;\R^{d})\cap V \right) \times \left( C^{\infty}_{c}(\Lambda;\R^{d})\cap V \right) \rightarrow H; \\
&F(u,v):=-P_{H}\left[(u\cdot\nabla)v\right],\ F(u):=F(u,u);
\end{align*}
can be extended to the well defined operators:
\[
\mathcal{N}:V\rightarrow V^{*};\ \mathcal{L}:V\rightarrow V^{*} ;\  F:V\times V\rightarrow V^{*}.
\]

With these preparations, we can write the model in the abstract form
\begin{equation}
dX_{t}=(\mathcal{N}(X_{t})+\mathcal{L}X_{t} + F(X_{t})+f)dt+dN_{t},
\label{LZM}
\end{equation}
where $N_{t}$ is a two-sided L\'{e}vy-process satisfying the condition $(N)$.
We then have the following result:
\begin{example}(Ladyzhenskaya model) Let $d \leq 6$. Then there exists a $p_{c} = p_{c}(d) > 2$ such that for $p \in (1,p_{c}]$ there is a continuous cocycle and a random attractor associated to \eqref{LZM}.
\end{example}
\begin{proof}
	We note the following properties of $\tau^{\mathcal{N}}$ \cite[pp. 198, Lemma 1.19]{MNRR}:
	\begin{align}
		 (\tau^{\mathcal{N}}_{ij}(e(u)) - \tau^{\mathcal{N}}_{ij}(e(v)))(e_{ij}(u) - e_{ij}(v)) &\geq 0; \label{Lady_eq_LM} \\
		\tau^{\mathcal{N}}_{ij}(e(u)) e_{ij}(u) &\geq 0; \label{Lady_eq_coerc} \\
		|\tau^{\mathcal{N}}_{ij}(e(u))| &\leq C (1 + |e(u)|)^{p-1}. \label{Lady_eq_growth}
	\end{align}
	Furthermore, we need the following higher-order version of Korn's inequality (a proof can be found at the end of this section):
	\begin{equation}\label{Lady_eq_Korn}
			\| \nabla e(u) \|_{L^{2}} \geq C \| u \|_{H^{2,2}} \quad \forall u \in W_{0}^{2,2}(\Lambda; \mathbb{R}^{d}).
	\end{equation}
	The condition $(A1)$ is clear. For $(A2)$ we have to estimate three terms:
	\begin{enumerate}[(a)]
		\item $_{V^{*}} \langle \mathcal{N}(u) - \mathcal{N}(v), u - v \rangle_{V} = \langle \tau^{\mathcal{N}}(e(u)) - \tau^{\mathcal{N}}(e(v)) , e(u) - e(v) \rangle_{H} \leq 0$ by \eqref{Lady_eq_LM}.
		\item In this case we get by \eqref{Lady_eq_Korn}
		\begin{align*}
			_{V^{*}}\langle \mathcal{L}(u-v), u - v \rangle_{V} = - 2 \mu_{1} \| \nabla e(u-v) \|_{L^{2}}^{2} \leq - C \| u - v \|_{H^{2,2}}^{2}.
		\end{align*}
		\item We estimate
		\begin{align*}
			&_{V^{*}} \langle F(u) - F(v), u - v \rangle_{V} =\ _{V^{*}} \langle F(u - v, v) , u - v \rangle_{V} \\
			&\leq C \| \nabla v \|_{L^{q}} \| u - v \|_{L^{\frac{2q}{q-1}}}^{2} \leq C \| \nabla v \|_{L^{q}}  \| u - v \|_{V}^{2 \theta}  \| u - v \|_{H}^{2(1-\theta)} \\
			&\leq \varepsilon \| u - v \|_{V}^{2} + C_{\varepsilon} \| \nabla v \|_{L^{q}}^{\nu} \| u - v \|_{H}^{2},
		\end{align*}
		where we applied the Gagliardo-Nirenberg interpolation inequality \eqref{GN_inequality} as well as Young's inequality. Here the exponents $\theta$ and $\gamma$ are defined by
		\begin{align*}
			\theta = \frac{d}{4q}, \quad \nu = \frac{1}{1-\theta}, \quad \nu' = \frac{\nu}{\nu - 1} = \frac{1}{\theta}.
		\end{align*}
		For the above calculations to work, we need to have
		\begin{align*}
			\frac{d}{4q} = \theta \in (0,1) \quad \mathrm{and} \quad q > 1 \Leftrightarrow q \in \left(\frac{d}{4} \vee 1, \infty \right),
		\end{align*}
		On the other hand, for the term $\| \nabla v \|_{L^{q}}$ to be bounded, we need the Sobolev embedding $H^{2,2} \subset H^{1,q}$ which holds only if
		\begin{align*}
			2 - \frac{d}{2} \geq 1 - \frac{d}{q} \Leftrightarrow q \leq \frac{2d}{d-2}.
		\end{align*}
		Furthermore, to check \eqref{conditon on eta and rho}, we have to interpolate once more:
	\begin{align*}
		\| \nabla v \|_{L^{q}} \leq \| v \|_{H^{2,2}}^{\theta} \| v \|_{L^{2}}^{1-\theta},
	\end{align*}
	which implies
	$$
		\theta = \frac{qd + 2q - 2d}{4q}.
	$$
	The condition $\theta \in [\frac{1}{2} , 1)$ from \eqref{GN_inequality} implies $q \geq 2$ and the condition $\nu \theta = \frac{4q}{4q-d} \theta \leq 2$ implies $d \leq 6$.
	
		Thus, in total we have to have
		\begin{align*}
			q \in \left( \frac{d}{4} \vee 2, \frac{2d}{d-2} \right],
		\end{align*}
		which is nonempty for $1 < d < 10$.
	\end{enumerate}
	Putting the three estimates together we find
	\begin{align*}
		&_{V^{*}} \langle \mathcal{N}(u) + \mathcal{L}u + F(u) - \mathcal{N}(v) - \mathcal{L}v - F(v), u - v \rangle_{V} \\
		&\leq - (C - \varepsilon) \| u - v \|_{H^{2,2}}^{2} + C_{\varepsilon} \| \nabla v \|_{L^{q}}^{\frac{4q}{4q-d}} \| u - v \|_{H}^{2},
	\end{align*}
	i.e. $(A2)$ with $\rho(v) = C_{\varepsilon} \| \nabla v \|_{L^{q}}^{\frac{4q}{4q-d}}$ and $\eta = 0$. By the choice of $q$ and the Sobolev embedding theorem, $\rho$ is locally bounded.
	
	For assumption $(A3)$ we proceed in a similar fashion (by the incompressibility condition, the term involving $F$ is zero):
	\begin{enumerate}[(a)]
		\item $_{V^{*}} \langle \mathcal{N}(v),v \rangle_{V} = - \langle \tau^{\mathcal{N}}(e(v)) , e(v) \rangle_{H} \leq 0$ by \eqref{Lady_eq_coerc}.
		\item $_{V^{*}} \langle \mathcal{L}(v),v \rangle_{V} = - \| \nabla e(v) \|_{L^{2}}^{2} \leq - C_{1} \| v \|_{H^{2,2}}^{2} = - C_{1} \| v \|_{V}^{2} $,
	\end{enumerate}
	and thus $(A3)$ holds with $\alpha = 2$. Here we have again the case that the constant $K$ in $(A3)$ vanishes, thus the condition $K < \frac{\lambda C_{1}}{4}$ is trivially satisfied.
	
	Note that up to this point, the parameter $p$ did not appear in any of the calculations.
	
	Assumption $(A4)$ requires to calculate three terms again:
	\begin{enumerate}[(a)]
		\item For the term $\mathcal{N}$, we distinguish two cases:
		\newline
		(i) Let $1 < p \leq 2$. By \eqref{Lady_eq_growth} we find
		\begin{align*}
			| ~ _{V^{*}} \langle \mathcal{N}(v), u \rangle_{V} |
			&\leq \int_{\Lambda} | \tau^{\mathcal{N}}(e(v)) | |e(u)| dx
			\leq C \int_{\Lambda} (1 + |e(v)|)^{p-1}  |e(u)| dx \\			
			&\leq C (1 + \| e(v) \|_{L^{p}}^{p-1}) \| e(u) \|_{L^{p}} \leq C (1 + \| v \|_{V}^{p-1}) \| u \|_{V} \\
			&\leq C(1 + \| v \|_{V}) \| u \|_{V}.
		\end{align*}
		(ii) Now let $p > 2$. Again, applying \eqref{Lady_eq_growth} we get
		\begin{align*}
			| ~ _{V^{*}} \langle \mathcal{N}(v), u \rangle_{V} |
			&\leq \int_{\Lambda} | \tau^{\mathcal{N}}(e(v)) | |e(u)| dx
			\leq C \int_{\Lambda} (1 + |e(v)|)^{p-1}  |e(u)| dx \\			
			&\leq C (1 + \| e(v) \|_{L^{p}}^{p-1}) \| e(u) \|_{L^{p}} \leq C ( 1 + \| v \|_{H^{1,p}}^{p-1}) \| u \|_{H^{1,p}} \\
			&\leq C \left( 1 + \| v \|_{V}^{\theta (p-1)} \| v \|_{H}^{(1 - \theta) (p-1)} \right) \| u \|_{V}
		\end{align*}
		where we used the Sobolev embedding $V = H^{2,2} \subset H^{1,p}$ which holds for $p \leq \frac{2d}{d-2}$ and the Gagliardo-Nirenberg inequality \eqref{GN_inequality} with
		$$
			\theta = \frac{dp + 2p - 2d}{4p},
		$$
		which has to be in $\left[ \frac{1}{2},1 \right)$. However, since $\alpha = 2$, we need that $\theta(p-1) \leq 1$. As long as $p \leq 2$ this condition is always satisfied. For $p > 2$ this is more difficult. We want to have
		\begin{align*}
			1 \geq \theta (p-1) \Leftrightarrow 0 \geq p^{2} - 3p + \frac{2d}{d+2}.
		\end{align*}
		We see that the latter condition is always strictly satisfied for $p=2$ but never satisfied for $p=3$. The critical value of $p$ can be calculated as
		\begin{align*}
			p_{c} = \frac{3}{2} + \frac{1}{2}\sqrt{\frac{d+18}{d+2}}.
		\end{align*}
		As $d>1$ we find that $p_{c} < 2.618$. As $d \leq 6$ we find $p_{c} \geq \frac{3}{2} + \frac{1}{2} \sqrt{\frac{24}{8}} \approx 2.36$.
		
	This leaves us with two conditions for this range of $p$:
	$$
		2 < p \leq \frac{2d}{d-2} \wedge p_{c} = p_{c}.
	$$
		\item $| ~ _{V^{*}} \langle \mathcal{L}v, u \rangle_{V} | = | \langle \nabla e(v) ,\nabla e(u) \rangle_{L^{2}}| \leq \| \nabla e(v) \|_{L^{2}} \| \nabla e(u) \|_{L^{2}}  \leq \| v \|_{V} \| u \|_{V} $
		\item For the last term we find
		\begin{align*}
			| \langle F(v,v), u \rangle | = | \langle F(v,u),v \rangle | \leq C \| \nabla u \|_{L^{q}} \| v \|_{L^{\frac{2q}{q-1}}}^{2} \leq C \| u \|_{V} \| v \|_{V}^{2\theta} \| v \|_{H}^{2 (1-\theta)}
		\end{align*}
				where we have taken the biggest possible $q$, $q = \frac{2d}{d-2}$, and where $\theta = \frac{d-2}{8}$ and since $\alpha = 2$ we again need to have
		$$
			2 \theta \leq 1 \Leftrightarrow d \leq 6.
		$$
	\end{enumerate}
	The conditions $(V)$ and $(A5)$ are easily seen to be satisfied.
\end{proof}

\begin{proof}[Proof of \eqref{Lady_eq_Korn}]
	 The classical Korn inequality states that
	 \begin{equation}
	 	\int_{\Lambda} | e(u) |^{2} dx \geq C \| u \|_{H^{1,2}}^{2} \quad \forall u \in H_{0}^{1,2}(\Lambda; \mathbb{R}^{d}). \label{Lady_eq_classKorn}
	 \end{equation}
	 We would like to set $u = \nabla v$ for $v \in H_{0}^{2,2}(\Lambda; \mathbb{R}^{d})$. Note that
	 \begin{align*}
	 	\left( \nabla e(v) \right)_{k} = \left( \partial_{k} e_{ij}(v) \right)_{i,j=1}^{d} = \frac{1}{2} \left(\partial_{i} (\partial_{k}v)_{j} + \partial_{j} (\partial_{k}v)_{i}  \right) = e(\partial_{k} v).
	 \end{align*}
	 Now by applying \eqref{Lady_eq_classKorn} to the vector $\partial_{k}v$ for fixed $k$, we find
	 \begin{align*}
	 	\int_{\Lambda} | \nabla e(v) |^{2} dx &= \sum_{k} \int_{\Lambda} | \partial_{k} e(v) |^{2} dx = \sum_{k} \int_{\Lambda} | e(\partial_{k} v) |^{2} dx \\
	 	&\geq C \sum_{k} \| \partial_{k} v \|_{H^{1,2}}^{2} = C \sum_{k} \sum_{i,j} \| \partial_{i} \partial_{k} v_{j} \|_{L^{2}}^{2} \\
	 	&= C \sum_{k} \sum_{i} \| \partial_{i} \partial_{k} v \|_{L^{2}(\Lambda; \mathbb{R}^{d})}^{2}
	 	= C \| v \|_{H^{2,2}(\Lambda; \mathbb{R}^{d})}^{2}.
	 \end{align*}
\end{proof}

\subsection{Stochastic Cahn-Hilliard type equations}\label{sub:CH}

The Cahn-Hilliard equation is a classical model to describe phase separation in a binary alloy. The reader is referred to Novick-Cohen \cite{N98} for a survey of the classical Cahn-Hilliard equation (see also Da Prato, Debussche \cite{DD96} and Elezovi\'{c}, Mikeli\'{c} \cite{EM91} for the stochastic case) and to \cite{NC90} for Cahn-Hilliard type equations. Let $d\le3$. We want to study stochastic
Cahn-Hilliard type equations of the following form:
\begin{equation}
\begin{split} & dX= \left( -\Delta^{2}X+\Delta\varphi(X) \right) dt + dN_{t},\ X(0)=X_{0},\\
 & \nabla X\cdot n=\nabla(\Delta X)\cdot n=0\ \ \text{on}\ \ \partial\Lambda,
\end{split}
\label{Cahn-Hilliard}
\end{equation}
where $X$ is a scalar function, $N_{t}$ is an $L^{2}(\Lambda)$-valued, two-sided L\'{e}vy process satisfying condition $(N)$, and the nonlinearity $\varphi$ is a function that will be specified below. Let
\[
V_{0}:=\{u\in H^{4,2}(\Lambda):\ \nabla u\cdot n=\nabla(\Delta u)\cdot n=0\ \text{on}\ \partial\Lambda\},
\]
where $H^{4,2}(\Lambda)$ denotes the standard Sobolev space on $\Lambda$ (with values in $\R$).

We consider the following Gelfand triple
\[
V\subset H:=L^{2}(\Lambda)\subset V^{*},
\]
where
\[
V:=\text{completion of}\ \ V_{0}\ \ w.r.t.\ \|\cdot\|_{H^{2,2}}.
\]
Recall that we use the following (equivalent) Sobolev norm on $H^{2,2}$:
\[
\|u\|_{H^{2,2}}:=\left(\int_{\Lambda}|\Delta u|^{2}dx\right)^{1/2}.
\]
Then we get the following result for (\ref{Cahn-Hilliard}).
\begin{example}(Stochastic Cahn-Hilliard type equations)
Suppose that $\varphi\in C^{1}(\R)$ and there exist some positive constants $C$ and $p\le 2$ such that
\begin{equation}\label{eq:CH_vphi}
	\begin{split}
 & \varphi^{\prime}(x)\ge-C_{\varphi},\ |\varphi(x)|\le C(1+|x|^{p}),\ x\in\R;\\
 & |\varphi(x)-\varphi(y)|\le C(1+|x|^{p-1}+|y|^{p-1})|x-y|,\ x,y\in\R.
 	\end{split}
\end{equation}
Let $C_{GN}$ be the constant from the Gagliardo-Nirenberg interpolation inequality \eqref{GN_inequality} for $H^{1,2}(\Lambda) \subset H^{2,2}(\Lambda) \cap L^{2}(\Lambda)$ and $\lambda$ the constant from the embedding $V \subset H$. Assume that $C_{\varphi} < \frac{\sqrt{\lambda}}{2 C_{GN}^{2}}$.

Then there exists a continuous cocycle and a random attractor associated to $(\ref{Cahn-Hilliard})$. \end{example}
\begin{proof}
We denote
\[
A(u):=-\Delta^{2}u+\Delta\varphi(u),\ u\in H^{4,2}(\Lambda).
\]
Note that for $u\in V_{0}$ by Sobolev's inequality (the embedding $V \subset W^{d,1} \subset L^{\infty}$ holds by our assumption on the dimension $d$) we have
\[
\begin{split}\left|_{V^{*}}\<A(u),v\>_{V}\right| & =\left|\<-\Delta u+\varphi(u),\Delta v\>_{L^{2}}\right|  \le\|v\|_{V}\left(\|u\|_{V}+\|\varphi(u)\|_{L^{2}}\right)\\
 & \le C\|v\|_{V}\left(1+\|u\|_{V}+\|u\|_{L^{\infty}}^{p}\right) \le C\|v\|_{V}\left(1+\|u\|_{V}+\|u\|_{V}^{p}\right),\ v\in V.
\end{split}
\]
Therefore, by  continuity $A$ can be extended to a map from $V$ to $V^{*}$. Moreover, this also implies that $A$ is hemicontinuous, i.e. $(A1)$ holds.

The other conditions $(A2)$--$(A4)$ as well as \eqref{conditon on eta and rho} were shown in \cite[Example 3.3]{LR12} with $\alpha = 2$, $\beta = 2 (p-1)$. As we need the exact form of the coercivity condition $(A3)$ to check the condition $K < \frac{\gamma \lambda}{4}$, we will repeat its proof.
By the interpolation inequality (\ref{GN_inequality}) and Young's inequality we have for any $v\in V$,
\begin{align*}
_{V^{*}}\<\Delta\varphi(v),v\>_{V} &=-\int_{\Lambda}\varphi^{\prime}(v)|\nabla v|^{2}dx\le C_{\varphi} \|v\|_{H^{1,2}}^{2} \le C_{\varphi} C_{GN}^{2} \| v \|_{V} \| v \|_{H} \\
&\le\frac{1}{2}\|v\|_{V}^{2}+ \frac{1}{2} C_{\varphi}^{2} C_{GN}^{4} \|v\|_{H}^{2},
\end{align*}
i.e. $(A3)$ holds with $\alpha=2$ and $K = \frac{1}{2} C_{\varphi}^{2} C_{GN}^{4}$ and $\gamma = \frac{1}{2}$. Thus by our assumption on $C_{\varphi}$, the inequality $K < \frac{\gamma \lambda}{4} = \frac{\lambda}{8}$ holds. The condition $(V)$ is satisfied as the operator $M := - \Delta^{2}$ is strongly monotone. $(A5)$ and \eqref{conditon on eta and rho} are clearly satisfied as well.
\end{proof}

\begin{rem}
	(1) Note that the technical constraint $\beta (\alpha - 1) = 2 (p-1) \leq 2$ forces $p \leq 2$, so the method does not cover the "classical" Cahn-Hilliard equation for which $\varphi$ is a double-well potential, $\varphi(u) = u^{3} - u$, $i.e.$ $p=3$.
	
	(2) The results of this article on existence of a random attractor for stochastic Cahn-Hilliard type equations seem  not have been established in the literature before.
	
\end{rem}

\subsection{Stochastic Kuramoto-Sivashinsky equation}\label{sub:KS}
	The Kuramoto-Sivashinsky equation combines features of the Burgers equation with the Cahn-Hilliard type equations studied in the previous section. It was introduced in the works of Kuramoto \cite{Kuramoto78} and Michelson and Sivashinsky \cite{MS77, Sivashinsky80} as a model for flame propagation. The equation  in one spatial dimension  has the form
	\begin{equation}
		\partial_{t} u = -\partial_{x}^{4}u - \partial_{x}^{2} u - u \partial_{x} u.
	\end{equation}
	The first two terms on the right-hand side are of Cahn-Hilliard type (with $\varphi(x) = x$), the last term is of Burgers type. We will briefly show the existence of a continuous cocycle as well as a random attractor in the periodic case for a slightly generalized model.
	\begin{example}[Stochastic Kuramoto-Sivashinsky equation] Let $\Lambda = (-L,L)$, $L > 0$ and $p \le 2$. Let $\varphi \in C^{1}(\mathbb{R})$ satisfy the conditions \eqref{eq:CH_vphi} as well as $C_{\varphi} < \frac{\sqrt{\lambda}}{2C_{GN}^{2}}$, where $C_{GN}$ is as in Section \ref{sub:CH}. Furthermore, let $N_{t}$ be an $H$-valued two-sided L\'{e}vy process satisfying condition $(N)$. The space $H$ will be defined below.
	
	Then the equation
	\begin{equation}\label{eq:SKS}
		d u = \left( -\partial_{x}^{4}u - \partial_{x}^{2} \varphi(u) - u \partial_{x} u \right) dt + dN_{t}
	\end{equation}	
 with boundary conditions
	\begin{align*}
		\partial_{x}^{i} u(-L,t) = \partial_{x}^{i} u(L,t), \quad i = 0,\ldots,3
	\end{align*}
	and initial condition $u(x,0) = u_{0}(x), x \in \Lambda$ generates a continuous cocycle and has a random attractor.
	\end{example}
\begin{proof}
	 Let
	 $$
	 	H = \left\{ u \in L^{2}(\Lambda) :\ \int_{\Lambda}u(x) dx = 0 \right\}, \quad V = H^{2}_{per} \cap H.
	 $$
	 We write
	 \begin{align*}
	 	A(u) := \mathcal{L}u + \mathcal{N}(u) + \mathcal{B}(u) := - \partial_{x}^{4} u + \partial_{x}^{2} \varphi (u) - u \partial_{x}u, \quad u \in H^{4,2}(\Lambda).
	 \end{align*}
	 $\mathcal{L}$, $\mathcal{N}$ have been extended to operators from $V$ to $V^{*}$ in Section \ref{sub:CH}, where the conditions $(A1)$--$(A4)$ were checked for them as well. That $\mathcal{B}$ is well-defined can be seen from the following calculations: by \eqref{GN_inequality} we find
	 \begin{align*}
	 	\left|~_{V^{*}}\langle u \partial_{x} u, v \rangle_{V} \right| = \frac{1}{2} \left| \langle \partial_{x} (u^{2}), v \rangle_{L^{2}} \right| \leq \frac{1}{2} \| u \|_{L^{4}}^{2} \| \partial_{x} v \|_{L^{2}} \leq C \| u \|_{V} \| u \|_{H} \| v \|_{V}.
	 \end{align*}
This not only implies the extendability but also gives the remaining contribution to $(A1)$ as well as to $(A4)$ with $\alpha = 2$, $\beta = 2$. For the local monotonicity we note that by the embeddings $H^{2,2} \subseteq H^{1,2} \subseteq W^{1,1} \subseteq L^{\infty}$, we find
\begin{align*}
	2~_{V^{*}}\langle u \partial_{x} u - v \partial_{x} v, u - v \rangle_{V} =2 \int_{\Lambda} (u-v)^{2} \partial_{x} v d x \leq C \| \partial_{x} v \|_{L^{\infty}} \| u - v \|_{H}^{2}
\end{align*}
which gives $(A2)$ with another locally bounded contribution $\rho_{\mathcal{B}}(v) = \|\partial_{x} v \|_{L^{\infty}}$. For $(A3)$ we note that $_{V^{*}}\langle \mathcal{B}(v),v \rangle_{V} = 0$. Thus the conditions $(A1)$--$(A4)$ are satisfied with $\alpha = \beta = 2$. The conditions $(A5)$, $(V)$ and \eqref{conditon on eta and rho} are again clearly satisfied.
\end{proof}
\begin{rem}
	 Yang \cite{Yang06} has studied stochastic Kuramoto-Sivashinsky equation in the case $d=1$, $\varphi(x) = -x$ with periodic boundary conditions and proved the existence of a random attractor for $H$-valued trace-class Wiener noise. The above result extends this to a more general class of equations and also to the case of L\'{e}vy noise.
\end{rem}

\subsection{SPDE with monotone coefficients} \label{sub:monotone}

In \cite{G13}, the stochastic evolution equation
\begin{align*}
	dX_{t} = A(t, X_{t}) dt + dW_{t} + \mu X_{t} \circ d \beta_{t}
\end{align*}
is considered on a Gelfand triple $V \subseteq H \subseteq V^{*}$, where the Wiener process takes values in $H$, $\mu \in \R$, $\beta_{t}$ is a real-valued Brownian motion and $\circ$ denotes Stratonovich integration. The operator $A$ in this context satisfies $(A1)$, $(A2)$ with $\rho = \eta = 0$, $(A3)$, and $(A4)$ with $\beta = 0$ and coefficients $C, \gamma, K$ depending on $(t,\omega)$. This case of a ``globally" monotone operator (typically just called {\it monotone operator}) is covered by the theorems in this work, if $\mu = 0$ and the coefficients $C, \gamma, K$ are independent of $(t,\omega)$ and satisfy $K < \frac{\gamma \lambda}{4}$.
Note that $\beta (\alpha - 1) = 0 \leq 2$ is satisfied in this case, and so is \eqref{conditon on eta and rho}.

Accordingly, all examples considered in \cite{G13} under these assumptions are covered by the results of this paper. These examples include the stochastic generalized $p$-Laplace equations on a Riemannian manifold, stochastic reaction diffusion equations, the stochastic porous media equation as well as the stochastic $p$-Laplace type equations studied by Zhao and Li \cite{ZL11} and the degenerate semilinear parabolic equation considered by Yang and Kloeden in \cite{YK11}. For more details the reader is referred to \cite{G13} and the references therein.
\appendix

\section{Existence and uniqueness of solutions to locally monotone PDE}\label{app:var_spde}

In this section we recall an existence and uniqueness result for locally monotone PDE (cf.\ \cite{L11,LR12,LR15}).
As before, let $V\subseteq H\subseteq V^{*}$ be a Gelfand triple. We consider the following general nonlinear evolution equation
\begin{align}
u'(t) & =A(t,u(t)),\quad\forall0<t<T,\label{eq:app_e_e}\\
u(0) & =u_{0}\in H,\nonumber
\end{align}
where $T>0$, $u'$ is the generalized derivative of $u$ on $(0,T)$ and $A:[0,T]\times V\rightarrow V^{*}$ is restrictedly measurable, $i.e.$ for each
$dt$-version of $u\in L^{1}([0,T];V)$, $t\mapsto A(\cdot,u(\cdot))$ is $V^{*}$-measurable on $[0,T]$.

Suppose that for some $\alpha>1,\beta\ge0$ there exist constants $c>0$, $C \ge 0$ and functions $f,g\in L^{1}([0,T];\mathbb{R})$ such that the following
conditions hold for all $t\in[0,T]$ and $v,v_{1},v_{2}\in V$:
\begin{enumerate}
\item $(H1)$ (Hemicontinuity) The map $s\mapsto{}_{V^{*}}\<A(t,v_{1}+sv_{2}),v\>_{V}$ is continuous on $\mathbb{R}$.
\item $(H2)$ (Local monotonicity)
\[
2{}_{V^{*}}\<A(t,v_{1})-A(t,v_{2}),v_{1}-v_{2}\>_{V}\le\left(f(t)+\eta(v_{1})+\rho(v_{2})\right)\|v_{1}-v_{2}\|_{H}^{2},
\]
where $\eta,\rho:V\rightarrow[0,+\infty)$ are measurable and locally bounded functions.
\item $(H3)$ (Generalized coercivity)
\[
2{}_{V^{*}}\<A(t,v),v\>_{V}\le-c\|v\|_{V}^{\alpha}+g(t)\|v\|_{H}^{2}+f(t).
\]

\item $(H4)$ (Growth)
\[
\|A(t,v)\|_{V^{*}}^{\frac{\alpha}{\alpha-1}}\le\bigg(f(t)+C\|v\|_{V}^{\alpha}\bigg)\bigg(1+\|v\|_{H}^{\beta}\bigg).
\]
\end{enumerate}
\begin{thm}
\label{thm:var_ex} Suppose that $V\subseteq H$ is compact and $(H1)$--$(H4)$ hold. Then for any $u_{0}\in H$, \eqref{eq:app_e_e} has a solution $u$ on
$[0,T]$, i.e.
\[
u\in L^{\alpha}([0,T];V)\cap C([0,T];H),\ u'\in L^{\frac{\alpha}{\alpha-1}}([0,T];V^{*})
\]
and
\[
\<u(t),v\>_{H}=\<u_{0},v\>_{H}+\int_{0}^{t}{}_{V^{*}}\<A(s,u(s)),v\>_{V}ds,\quad\forall t\in[0,T],v\in V.
\]
Moreover, if there exist non-negative constants $C$, $\gamma$ such that
\begin{equation}
\eta(v)+\rho(v)\le C(1+\|v\|_{V}^{\alpha})(1+\|v\|_{H}^{\gamma}),\ v\in V,\label{c3}
\end{equation}
then the solution of \eqref{eq:app_e_e} is unique.\end{thm}
\begin{proof}
The conclusions follow from a more general result in \cite{LR12} (see Theorem 1.1 and Remark 1.1(3)) or \cite[Theorem 5.2.2]{LR15}.
\end{proof}

\section{Stochastic Flows and RDS}\label{app:rds}

We recall the framework of stochastic flows, random dynamical system (RDS) and random attractors. For more details we refer to \cite{A98,CDF97,CF94,S92}.
Let $(H,d)$ be a complete separable metric space and $(\O,\F,\P,\{\t_{t}\}_{t\in\R})$ be a metric dynamical system, i.e.\ $(t,\o)\mapsto\theta_{t}(\o)$
is ($\mcB(\R)\otimes\mcF,\mcF)$-measurable, $\theta_{0}=$ id, $\theta_{t+s}=\theta_{t}\circ\theta_{s}$ and $\theta_{t}$ is $\P$-preserving for all
$s,t\in\R$.
\begin{defn}
A family of maps $S(t,s;\o):H\to H$, $s\le t$ is said to be a stochastic flow, if for every $\o\in\O$
\begin{enumerate}
\item [i.] $S(s,s;\o)=id_{H}$, for all $s\in\R$.
\item [ii.] $S(t,s;\o)x=S(t,r;\o)S(r,s;\o)x$, for all $t\ge r\ge s$, $x\in H$.
\end{enumerate}
A stochastic flow $S(t,s;\o)$ is called
\begin{enumerate}
\item [iii.]measurable if $(t,s,\o,x)\to S(t,s;\o)x$ is measurable.
\item [iv.]continuous if $x\mapsto S(t,s;\o)x$ is continuous for all $s\le t$, $\o\in\O$.
\item [v.]a cocycle if $S(t,s;\o)x=S(t-s,0;\t_{s}\o)x,$ for all $x\in H$, $t\ge s$, $\o\in\O$.
\end{enumerate}
A measurable, cocycle stochastic flow is also called a random dynamical system (RDS).
\end{defn}
For a cocycle stochastic flow the notation of the initial time $s\in\R$ is redundant. Therefore, often the notation $\vp(t,\o):=S(t,0;\o)$ is chosen for
cocycles in the literature. Since all the results may be extended to a time-inhomogeneous setup (where $S(t,s;\o)$ is not a cocycle in general) we prefer
to use the notation $S(t,s;\o)$.%

\begin{defn}
A function $f : \R\to\R_{+}$ is said to be
\begin{enumerate}
\item tempered if $\lim\limits _{r\to-\infty}f_{r}e^{\eta r}=0$ for all $\eta > 0$;
\item exponentially integrable if $ f \in L_{loc}^{1}(\R;\R_{+})$ and $\int_{-\infty}^{t}f_{r}e^{\eta r}dr < \infty$ for all $t\in\R$, $\eta > 0$.
\end{enumerate}
\end{defn}
Let us note that the product of two tempered functions is tempered and that the product of a tempered and an exponentially integrable function is exponentially
integrable if it is locally integrable.

In the following, let $S(t,s;\o)$ be a cocycle.
\begin{defn}
A family $\{D(\o)\}_{\o\in\O}$ of subsets of $H$ is said to be
\begin{enumerate}
\item a random closed set if it is $\P$-a.s.\ closed and $\o\to d(x,D(\o))$ is measurable for each $x\in H$. In this case we also call $D$ measurable.

\item tempered if $t\mapsto\|D(\theta_{t}\o)\|_{H}$ is a tempered function for all $\o\in\O$ (assuming $H$ to be a normed space).
\item strictly stationary if $D(t,\o) = D(0,\t_t\o)$ for all $\o \in \O$, $t \in \R$.
\end{enumerate}
\end{defn}
From now on let $\mcD$ be a system of families $\{D(\o)\}_{\o\in\O}$ of subsets of $H$. For two subsets $A,B\subseteq H$ we define
\[
d(A,B):=\begin{cases}
\sup\limits _{a\in A}\inf\limits _{b\in B}d(a,b), & \text{ if }A\ne\emptyset;\\
\infty, & \text{ otherwise.}
\end{cases}
\]

\begin{defn}
\label{def:abs} A family $\{K(\o)\}_{\o\in\O}$ of subsets of $H$ is said to be
\begin{enumerate}
\item $\mcD$-absorbing, if there exists an absorption time $s_{0}=s_{0}(\o,D)$ such that
\[
S(0,s;\o)D(\theta_{s}\o)\subseteq K(\o),\quad\forall s\le s_{0}
\]
for all $D\in\mcD$ and $\o\in\O_{0}$, where $\O_{0}\subseteq\O$ is a subset of full $\P$-measure.
\item $\mcD$-attracting, if
\[
d(S(0,s;\o)D(\theta_{s}\o),K(\o))\to0,\quad s\to-\infty
\]
for all $D\in\mcD$ and $\o\in\O_{0}$, where $\O_{0}\subseteq\O$ is a subset of full $\P$-measure.
\end{enumerate}
\end{defn}

\begin{defn}
\label{def:asympt_cmpt_flow} A cocycle $S(t,s;\o)$ is called
\begin{enumerate}
\item $\mcD$-asymptotically compact if there is a random, compact, $\mcD$-attracting set $\{K(\o)\}_{\o\in\O}$ .
\item compact if for all $t>s$, $\o\in\O$ and $B\subseteq H$ bounded, $S(t,s;\o)B$ is precompact in $H$.
\end{enumerate}
\end{defn}
We define the $\O$-limit set by
\[
\O(D;\o):=\bigcap_{r<0}\overline{\bigcup_{\tau<r}S(0,\tau;\o)D(\theta_{\tau}\o)},
\]
and one can show that (cf.\cite{CF94})
\[
\O(D;\o)=\{x\in H|\ \exists s_{n}\to-\infty,\ x_{n}\in D(\theta_{s_{n}}\o)\text{ such that }S(0,s_{n};\o)x_{n}\to x\}.
\]

\begin{defn}
Let $S(t,s;\o)$ be a cocycle. A random closed set $\{\mcA(\o)\}_{\o\in\O}$ is called a $\mcD$-random attractor for $S(t,s;\o)$ if it satisfies $\P$-a.s.
\begin{enumerate}
\item $\mcA(\o)$ is nonempty and compact.
\item $\mcA$ is $\mcD$-attracting.
\item $\mcA(\o)$ is invariant under $S(t,s;\o)$, i.e. for each $s\le t$
\[
S(t,s;\o)\mcA(\theta_{s}\o)=\mcA(\theta_{t}\o).
\]

\end{enumerate}
\end{defn}
The following theorem gives a sufficient condition for the existence of a random attractor (cf. e.g. \cite{CF94}). Let $o\in H$ be an arbitrary point in
$H$.
\begin{thm}
\label{thm:suff_cond_attr} Let $S(t,s;\o)$ be a continuous, $\mcD$-asymptotically compact cocycle and let $K$ be a corresponding random, compact,
$\mcD$-attracting set. Then
\[
\mcA(\o):=\begin{cases}
\overline{\bigcup_{D\in\mcD}\O(D;\o)} & ,\text{ if }\o\in\O_{0};\\
\{o\} & ,\text{ otherwise.}
\end{cases}
\]
defines a random $\mcD$-attractor for $S(t,s;\o)$ and $\mcA(\o)\subseteq K(\o)\cap\O(K;\o)$ for all $\o\in\O_{0}$ (where $\O_{0}$ is as in Definition
\ref{def:abs}).
\end{thm}
Now we introduce the notion of (stationary) conjugation mappings and conjugated stochastic flows (cf.\ \cite{K01,IL01}).
\begin{defn}
Let $(H,d)$ and $(\td H,\td d)$ be two metric spaces.
\begin{enumerate}
\item A family of homeomorphisms $\mcT=\{T(\o):H\to\td H\}_{\o\in\O}$ such that the maps $\o\mapsto T(\o)x$ and $\o\mapsto T^{-1}(\o)y$ are measurable
    for all $x\in H,y\in\td H$, is called a stationary conjugation mapping. We set $T(t,\o):=T(\t_{t}\o)$.
\item Let $Z(t,s;\o),S(t,s;\o)$ be cocycles. $Z(t,s;\o)$ and $S(t,s;\o)$ are said to be stationary conjugated, if there is a stationary conjugation
    mapping $\mcT$ such that
\[
S(t,s;\o)=T(t,\o)\circ Z(t,s;\o)\circ T^{-1}(s,\o).
\]

\end{enumerate}
\end{defn}
It is easy to show that stationary conjugation mappings preserve the stochastic flow and cocycle property.
\begin{prop}
\label{prop:def_conj_flow} Let $\mcT$ be a stationary conjugation mapping and $Z(t,s;\o)$ be a continuous cocycle. Then
\[
S(t,s;\o):=T(t,\o)\circ Z(t,s;\o)\circ T^{-1}(s,\o)
\]
defines a conjugated continuous cocycle.
\end{prop}
The existence of a random attractor is preserved under conjugation.
\begin{thm}
\label{thm:conj_attractor} Let $S(t,s;\o)$ and $Z(t,s;\o)$ be cocycles conjugated by a stationary conjugation mapping $\mcT$ consisting of uniformly
continuous mappings $T(\o):H\to H$. Assume that there is a $\tilde{\mcD}$-attractor $\tilde{\mcA}$ for $Z(t,s;\o)$ and let
\[
\mcD:=\big\{\{T(\o)\td D(\o)\}_{\o\in\O}|\ \td D\in\td\mcD\big\}.
\]
Then $\mcA(\o):=T(\o)\tilde{\mcA}(\o)$ is a random $\mcD$-attractor for $S(t,s;\o)$.
\end{thm}
We will require the following strong notion of stationarity:
\begin{defn}
\label{def:strict_stat} A map $X:\R\times\O\to H$ is said to satisfy (crude) strict stationarity, if
\[
X(t,\o)=X(0,\t_{t}\o)
\]
for all $\o\in\O$ and $t\in\R$ (for all $t\in\R$, $\P$-a.s., where the zero-set may depend on $t$ resp.).
\end{defn}
As $\P$ is $\t$-invariant, crude strict stationarity implies stationarity of the law. Objects obtained as limits in $L^{2}(\O)$ or limits in
probability usually only satisfy crude strict stationarity. Thus one needs the existence of selections of indistinguishable strictly stationary
versions. The following Proposition provides these and is an easy adaption of \cite[Proposition 2.8]{L01}.
\begin{prop}
\label{prop:stat_perfect} Let $V\subseteq H$ and $X:\R\times\O\to H$ be a process satisfying crude stationarity. Assume that $X_{\cdot}\in
D(\R;H)\cap L_{loc}^{\a}(\R;V)$ for some $\a\ge1$, $\P$-a.s. Then there exists a process $\tilde{X}:\R\times\O\to H$ such that
\begin{enumerate}
\item $\tilde{X}_{\cdot}(\o)\in D(\R;H)\cap L_{loc}^{\a}(\R;V)$ for all $\o\in\O$.
\item $X$, $\tilde{X}$ are indistinguishable, i.e.
\[
\P[X_{t}\ne\tilde{X}_{t}\text{ for some }t\in\R]=0,
\]
with a $\theta$-invariant exceptional set.
\item $\tilde{X}$ is strictly stationary.
\end{enumerate}
\end{prop}

\section*{Acknowledgements}
Supported in part by NSFC (11571147,11822106,11831014), NSF of Jiangsu Province
(BK20160004), and the PAPD of Jiangsu Higher Education Institutions, the Max-Planck Society through MPI MIS Leipzig, as well as the DFG through SFB-701, SFB-1283, IRTG 2235 and the BiBoS-Research Center.

The authors would like to thank Michael R\"ockner for valuable discussions and comments.

 \bibliographystyle{plain}
\bibliography{refs}

\def\ocirc#1{\ifmmode\setbox0=\hbox{$#1$}\dimen0=\ht0 \advance\dimen0
  by1pt\rlap{\hbox to\wd0{\hss\raise\dimen0
  \hbox{\hskip.2em$\scriptscriptstyle\circ$}\hss}}#1\else {\accent"17 #1}\fi}
\begin{thebibliography}{10}

\bibitem{AR05}
Sergio Albeverio and Barbara R{\"u}diger.
\newblock Stochastic integrals and the {L}\'evy-{I}to decomposition theorem on
  separable {B}anach spaces.
\newblock {\em Stoch. Anal. Appl.}, 23(2):217--253, 2005.

\bibitem{A09}
David Applebaum.
\newblock {\em L\'evy processes and stochastic calculus}, volume 116 of {\em
  Cambridge Studies in Advanced Mathematics}.
\newblock Cambridge University Press, Cambridge, second edition, 2009.

\bibitem{A98}
Ludwig Arnold.
\newblock {\em Random dynamical systems}.
\newblock Springer Monographs in Mathematics. Springer-Verlag, Berlin, 1998.

\bibitem{Bak07}
Yuri Bakhtin.
\newblock Burgers equation with random boundary conditions.
\newblock {\em Proc. Amer. Math. Soc.}, 135(7):2257--2262, 2007.

\bibitem{BCK14}
Yuri Bakhtin, Eric Cator, and Konstantin~M. Khanin.
\newblock Space-time stationary solutions for the {B}urgers equation.
\newblock {\em J. Amer. Math. Soc.}, 27(1):193--238, 2014.

\bibitem{BBF14}
David Barbato, Hakima Bessaih, and Benedetta Ferrario.
\newblock On a stochastic leray-{$\alpha$} model of euler equations.
\newblock {\em Stoch. Proc. Appl.}, 124(1):199--219, 2014.

\bibitem{BLL06}
Peter~W. Bates, Hannelore Lisei, and Kening Lu.
\newblock Attractors for stochastic lattice dynamical systems.
\newblock {\em Stoch. Dyn.}, 6(1):1--21, 2006.

\bibitem{BLW09}
Peter~W. Bates, Kening Lu, and Bixiang Wang.
\newblock Random attractors for stochastic reaction-diffusion equations on
  unbounded domains.
\newblock {\em J. Differential Equations}, 246(2):845--869, 2009.

\bibitem{BLW13}
Peter~W. Bates, Kening Lu, and Bixiang Wang.
\newblock Tempered random attractors for parabolic equations in weighted
  spaces.
\newblock {\em J. Math. Phys.}, 54(8):081505, 26, 2013.

\bibitem{BLW14}
Peter~W. Bates, Kening Lu, and Bixiang Wang.
\newblock Attractors of non-autonomous stochastic lattice systems in weighted
  spaces.
\newblock {\em Phys. D}, 289:32--50, 2014.

\bibitem{BT73}
Alain Bensoussan and Roger Temam.
\newblock Equations stochastiques du type navier-stokes.
\newblock {\em Journal of Functional Analysis}, 13(2):195--222, 1973.

\bibitem{BGLR10}
Wolf-J{\"u}rgen Beyn, Benjamin Gess, Paul Lescot, and Michael R{\"o}ckner.
\newblock The global random attractor for a class of stochastic porous media
  equations.
\newblock {\em Comm. Partial Differential Equations}, 36(3):446--469, 2011.

\bibitem{BL06}
Zdzislaw Brze{\'z}niak and Yuhong Li.
\newblock Asymptotic compactness and absorbing sets for 2d stochastic
  navier-stokes equations on some unbounded domains.
\newblock {\em Trans. Amer. Math. Soc.}, 358(12):5587--5629, 2006.

\bibitem{BLZ12}
Zdzislaw Brze{\'z}niak, Wei Liu, and Jiahui Zhu.
\newblock Strong solutions for {SPDE} with locally monotone coefficients driven
  by {L}evy noise.
\newblock {\em Nonlinear Anal. Real World Appl.}, 17:283--310, 2014.

\bibitem{CSY15}
Daomin Cao, Chunyou Sun, and Meihua Yang.
\newblock Dynamics for a stochastic reaction-diffusion equation with additive
  noise.
\newblock {\em J. Differential Equations}, 259(3):838--872, 2015.

\bibitem{CL03}
Tom{\'a}s Caraballo and Jos\'{e}~A. Langa.
\newblock On the upper semicontinuity of cocycle attractors for non-autonomous
  and random dynamical systems.
\newblock {\em Dyn. Contin. Discrete Impuls. Syst. Ser. A Math. Anal.},
  10(4):491--513, 2003.

\bibitem{CLR98}
Tom{\'a}s Caraballo, Jos{\'e}~A. Langa, and James~C. Robinson.
\newblock Upper semicontinuity of attractors for small random perturbations of
  dynamical systems.
\newblock {\em Comm. Partial Differential Equations}, 23(9-10):1557--1581,
  1998.

\bibitem{CH05}
Alexey Cheskidov, Darryl~D. Holm, Eric Olson, and Edriss~S. Titi.
\newblock On a {L}eray-{$\alpha$} model of turbulence.
\newblock {\em Proc. R. Soc. Lond. Ser. A Math. Phys. Eng. Sci.},
  461(2055):629--649, 2005.

\bibitem{CM10}
Igor~D. Chueshov and Annie Millet.
\newblock Stochastic 2{D} hydrodynamical type systems: well posedness and large
  deviations.
\newblock {\em Appl. Math. Optim.}, 61(3):379--420, 2010.

\bibitem{CDF97}
Hans Crauel, Arnaud Debussche, and Franco Flandoli.
\newblock Random attractors.
\newblock {\em J. Dynam. Differential Equations}, 9(2):307--341, 1997.

\bibitem{CDS09}
Hans Crauel, Georgi Dimitroff, and Michael Scheutzow.
\newblock Criteria for strong and weak random attractors.
\newblock {\em J. Dynam. Differential Equations}, 21(2):233--247, 2009.

\bibitem{CF94}
Hans Crauel and Franco Flandoli.
\newblock Attractors for random dynamical systems.
\newblock {\em Probab. Theory Related Fields}, 100(3):365--393, 1994.

\bibitem{CLL18}
Hongyong Cui, Jos\'{e}~A. Langa, and Yangrong Li.
\newblock Measurability of random attractors for quasi strong-to-weak
  continuous random dynamical systems.
\newblock {\em J. Dynam. Differential Equations}, 30(4):1873--1898, 2018.

\bibitem{DD96}
Giuseppe {Da Prato} and Arnaud Debussche.
\newblock Stochastic {C}ahn-{H}illiard equation.
\newblock {\em Nonlinear Anal.}, 26(2):241--263, 1996.

\bibitem{DPD07}
Giuseppe Da~Prato and Arnaud Debussche.
\newblock {$m$}-dissipativity of {K}olmogorov operators corresponding to
  {B}urgers equations with space-time white noise.
\newblock {\em Potential Anal.}, 26(1):31--55, 2007.

\bibitem{Debussche97}
Arnaud Debussche.
\newblock On the finite dimensionality of random attractors.
\newblock {\em Stochastic Anal. Appl.}, 15(4):473--491, 1997.

\bibitem{DRS12}
Gabriel Deugou\'{e}, Paul~Andr\'{e} Razafimandimby, and Mamadou Sango.
\newblock On the 3-{D} stochastic magnetohydrodynamic-{$\alpha$} model.
\newblock {\em Stochastic Process. Appl.}, 122(5):2211--2248, 2012.

\bibitem{DS10}
Gabriel Deugoue and Mamadou Sango.
\newblock On the strong solution for the 3{D} stochastic {L}eray-alpha model.
\newblock {\em Bound. Value Probl.}, 2010.

\bibitem{EKMS00}
Weinan E, Konstantin~M. Khanin, Alexander~E. Mazel, and Yakov~G. Sinai.
\newblock Invariant measures for {B}urgers equation with stochastic forcing.
\newblock {\em Ann. of Math. (2)}, 151(3):877--960, 2000.

\bibitem{EM91}
Neven Elezovi{\'c} and Andro Mikeli{\'c}.
\newblock On the stochastic {C}ahn-{H}illiard equation.
\newblock {\em Nonlinear Anal.}, 16(12):1169--1200, 1991.

\bibitem{Fan04}
Xiaoming Fan.
\newblock Random attractor for a damped sine-{G}ordon equation with white
  noise.
\newblock {\em Pacific J. Math.}, 216(1):63--76, 2004.

\bibitem{FGS17}
Franco Flandoli, Benjamin Gess, and Michael Scheutzow.
\newblock Synchronization by noise.
\newblock {\em Probab. Theory Related Fields}, 168(3-4):511--556, 2017.

\bibitem{FR}
Jens Frehse and Michael R{\ocirc{u}}{\v{z}}i{\v{c}}ka.
\newblock Non-homogeneous generalized {N}ewtonian fluids.
\newblock {\em Math. Z.}, 260(2):355--375, 2008.

\bibitem{G13}
Benjamin Gess.
\newblock Random {A}ttractors for {D}egenerate {S}tochastic {P}artial
  {D}ifferential {E}quations.
\newblock {\em J. Dynam. Differential Equations}, 25(1):121--157, 2013.

\bibitem{G13-4}
Benjamin Gess.
\newblock Random attractors for singular stochastic evolution equations.
\newblock {\em J. Differential Equations}, 255(3):524--559, 2013.

\bibitem{Gess14}
Benjamin Gess.
\newblock Random attractors for stochastic porous media equations perturbed by
  space-time linear multiplicative noise.
\newblock {\em Ann. Probab.}, 42(2):818--864, 2014.

\bibitem{GLR11}
Benjamin Gess, Wei Liu, and Michael R{\"o}ckner.
\newblock Random attractors for a class of stochastic partial differential
  equations driven by general additive noise.
\newblock {\em J. Differential Equations}, 251(4-5):1225--1253, 2011.

\bibitem{GT16}
Benjamin Gess and Jonas~M. T\"{o}lle.
\newblock Ergodicity and local limits for stochastic local and nonlocal
  {$p$}-{L}aplace equations.
\newblock {\em SIAM J. Math. Anal.}, 48(6):4094--4125, 2016.

\bibitem{GLWY18}
Anhui Gu, Dingshi Li, Bixiang Wang, and Han Yang.
\newblock Regularity of random attractors for fractional stochastic
  reaction-diffusion equations on {$\Bbb{R}^n$}.
\newblock {\em J. Differential Equations}, 264(12):7094--7137, 2018.

\bibitem{GL17}
Anhui Gu and Yangrong Li.
\newblock Dynamic behavior of stochastic {$p$}-{L}aplacian-type lattice
  equations.
\newblock {\em Stoch. Dyn.}, 17(5):1750040, 19, 2017.

\bibitem{GW18}
Anhui Gu and Bixiang Wang.
\newblock Asymptotic behavior of random {F}itzhugh-{N}agumo systems driven by
  colored noise.
\newblock {\em Discrete Contin. Dyn. Syst. Ser. B}, 23(4):1689--1720, 2018.

\bibitem{GG10}
Boling Guo and Chunxiao Guo.
\newblock Remark on random attractors of stochastic non-{N}ewtonian fluid.
\newblock {\em J. Partial Differ. Equ.}, 23(1):16--32, 2010.

\bibitem{IL01}
Peter Imkeller and Christian Lederer.
\newblock On the cohomology of flows of stochastic and random differential
  equations.
\newblock {\em Probab. Theory Related Fields}, 120(2):209--235, 2001.

\bibitem{IK03}
Renato Iturriaga and Konstantin~M. Khanin.
\newblock Burgers turbulence and random {L}agrangian systems.
\newblock {\em Comm. Math. Phys.}, 232(3):377--428, 2003.

\bibitem{MNRR}
Mirko~Rokyta Josef~M{\'a}lek, Jind{\v{r}}ich~Ne{\v{c}}as and Michael
  R{\ocirc{u}}{\v{z}}i{\v{c}}ka.
\newblock {\em Weak and measure-valued solutions to evolutionary {PDE}s},
  volume~13 of {\em Applied Mathematics and Mathematical Computation}.
\newblock Chapman \& Hall, London, 1996.

\bibitem{K01}
Hannes Keller.
\newblock {\em Attractors of second order stochastic differential equations}.
\newblock Logos Verlag Berlin, Berlin, 2002.

\bibitem{KW14}
Andrew Krause and Bixiang Wang.
\newblock Pullback attractors of non-autonomous stochastic degenerate parabolic
  equations on unbounded domains.
\newblock {\em J. Math. Anal. Appl.}, 417(2):1018--1038, 2014.

\bibitem{KS04}
Serge\u{\i}~B. Kuksin and Shirikyan~Armen R.
\newblock On random attractors for systems of mixing type.
\newblock {\em Funktsional. Anal. i Prilozhen.}, 38(1):34--46, 2004.

\bibitem{Kuramoto78}
Yoshiki Kuramoto.
\newblock Diffusion-induced chaos in reaction systems.
\newblock {\em Progress of Theoretical Physics Supplement}, 64:346--367, 1978.

\bibitem{L70}
Olga~A. Ladyzhenskaya.
\newblock {\em New equations for the description of the viscous incompressible
  fluids and solvability in large of the boundary value problems for them},
  volume~V of {\em Boundary Value Problems of Mathematical Physics}.
\newblock 1970.

\bibitem{Langa03}
Jos\'{e}~A. Langa.
\newblock Finite-dimensional limiting dynamics of random dynamical systems.
\newblock {\em Dyn. Syst.}, 18(1):57--68, 2003.

\bibitem{L01}
Christian Lederer.
\newblock Konjugation stochastischer und zuf{\"a}lliger station{\"a}rer
  {D}ifferentialgleichungen und eine {V}ersion des lokalen {S}atzes von
  {H}artman-{G}robman f{\"u}r stochastische {D}ifferentialgleichungen.
\newblock {\em PhD thesis}, 2001.

\bibitem{L34}
Jean Leray.
\newblock Sur le mouvement d'un liquide visqueux emplissant l'espace.
\newblock {\em Acta Math.}, 63(1):193--248, 1934.

\bibitem{LLL18}
Xiaojun Li, Xiliang Li, and Kening Lu.
\newblock Random attractors for stochastic parabolic equations with additive
  noise in weighted spaces.
\newblock {\em Commun. Pure Appl. Anal.}, 17(3):729--749, 2018.

\bibitem{LCL14}
Yangrong Li, Hongyong Cui, and Jia Li.
\newblock Upper semi-continuity and regularity of random attractors on
  {$p$}-times integrable spaces and applications.
\newblock {\em Nonlinear Anal.}, 109:33--44, 2014.

\bibitem{LGL15}
Yangrong Li, Anhui Gu, and Jia Li.
\newblock Existence and continuity of bi-spatial random attractors and
  application to stochastic semilinear {L}aplacian equations.
\newblock {\em J. Differential Equations}, 258(2):504--534, 2015.

\bibitem{LG08}
Yangrong Li and Boling Guo.
\newblock Random attractors for quasi-continuous random dynamical systems and
  applications to stochastic reaction-diffusion equations.
\newblock {\em J. Differential Equations}, 245(7):1775--1800, 2008.

\bibitem{LY16}
Yangrong Li and Jinyan Yin.
\newblock Existence, regularity and approximation of global attractors for
  weakly dissipative {$p$}-{L}aplace equations.
\newblock {\em Discrete Contin. Dyn. Syst. Ser. S}, 9(6):1939--1957, 2016.

\bibitem{LY16b}
Yangrong Li and Jinyan Yin.
\newblock A modified proof of pullback attractors in a {S}obolev space for
  stochastic {F}itz{H}ugh-{N}agumo equations.
\newblock {\em Discrete Contin. Dyn. Syst. Ser. B}, 21(4):1203--1223, 2016.

\bibitem{L11}
Wei Liu.
\newblock Existence and uniqueness of solutions to nonlinear evolution
  equations with locally monotone operators.
\newblock {\em Nonlinear Anal.}, 74:7543--7561, 2011.

\bibitem{LR10}
Wei Liu and Michael R{\"o}ckner.
\newblock S{PDE} in {H}ilbert space with locally monotone coefficients.
\newblock {\em J. Funct. Anal.}, 259(11):2902--2922, 2010.

\bibitem{LR12}
Wei Liu and Michael R{\"o}ckner.
\newblock Local and global well-posedness of {SPDE} with generalized coercivity
  conditions.
\newblock {\em J. Differential Equations}, 254:725--755, 2013.

\bibitem{LR15}
Wei Liu and Michael R{\"o}ckner.
\newblock {\em Stochastic Partial Differential Equations: An Introduction}.
\newblock Universitext. Springer, Heidelberg, 2015.

\bibitem{HQWZ19}
Hong Lu, Jiangang Qi, Bixiang Wang, and Mingji Zhang.
\newblock Random attractors for non-autonomous fractional stochastic parabolic
  equations on unbounded domains.
\newblock {\em Discrete Contin. Dyn. Syst.}, 39(2):683--706, 2019.

\bibitem{MS77}
Daniel~M. Michelson and Gregory~I. Sivashinsky.
\newblock Nonlinear analysis of hydrodynamic instability in laminar flames.
  {II}. {N}umerical experiments.
\newblock {\em Acta Astronaut.}, 4(11-12):1207--1221, 1977.

\bibitem{NC90}
Amy Novick-Cohen.
\newblock On {C}ahn-{H}illiard type equations.
\newblock {\em Nonlinear Anal.}, 15(9):797--814, 1990.

\bibitem{N98}
Amy Novick-Cohen.
\newblock The {C}ahn-{H}illiard equation: mathematical and modeling
  perspectives.
\newblock {\em Adv. Math. Sci. Appl.}, 8(2):965--985, 1998.

\bibitem{PZ07}
Szymon Peszat and Jerzy Zabczyk.
\newblock {\em Stochastic partial differential equations with {L}{\'e}vy
  noise}, volume 113 of {\em Encyclopedia of Mathematics and its Applications}.
\newblock Cambridge University Press, Cambridge, 2007.
\newblock An evolution equation approach.

\bibitem{S92}
Bj{\"o}rn Schmalfuss.
\newblock Backward cocycle and attractors of stochastic differential equations.
\newblock In V.~Reitmann, T.~Riedrich, and N.~Koksch, editors, {\em
  International Seminar on Applied Mathematics - Nonlinear Dynamics: Attractor
  Approximation and Global Behavior}, pages 185--192. Technische
  Universit{\"a}t Dresden, 1992.

\bibitem{STW19}
Yadong Shang, Jianjun~Paul Tian, and Bixiang Wang.
\newblock Asymptotic behavior of the stochastic {K}eller-{S}egel equations.
\newblock {\em Discrete Contin. Dyn. Syst. Ser. B}, 24(3):1367--1391, 2019.

\bibitem{Sivashinsky80}
Gregory~I. Sivashinsky.
\newblock On flame propagation under conditions of stoichiometry.
\newblock {\em SIAM J. Appl. Math.}, 39(1):67--82, 1980.

\bibitem{taira:95}
Kazuaki Taira.
\newblock {\em Analytic semigroups and semilinear initial boundary value
  problems}.
\newblock Cambridge University Press, 1995.

\bibitem{T01}
Roger Temam.
\newblock {\em Navier-{S}tokes equations}.
\newblock AMS Chelsea Publishing, Providence, RI, 2001.
\newblock Theory and numerical analysis, Reprint of the 1984 edition.

\bibitem{CTV07}
Mark~I. Vishik, Edriss~S. Titi, and Vladimir~V. Chepyzhov.
\newblock On the convergence of trajectory attractors of the three-dimensional
  {N}avier-{S}tokes {$\alpha$}-model as {$\alpha \to 0$}.
\newblock {\em Mat. Sb.}, 198(12):3--36, 2007.

\bibitem{Wang09}
Bixiang Wang.
\newblock Random attractors for the stochastic {B}enjamin-{B}ona-{M}ahony
  equation on unbounded domains.
\newblock {\em J. Differential Equations}, 246(6):2506--2537, 2009.

\bibitem{Wang14}
Bixiang Wang.
\newblock Existence and upper semicontinuity of attractors for stochastic
  equations with deterministic non-autonomous terms.
\newblock {\em Stoch. Dyn.}, 14(4), 2014.

\bibitem{Wang17}
Bixiang Wang.
\newblock Asymptotic behavior of non-autonomous fractional stochastic
  reaction-diffusion equations.
\newblock {\em Nonlinear Anal.}, 158:60--82, 2017.

\bibitem{Wang18}
Bixiang Wang.
\newblock Weak pullback attractors for mean random dynamical systems in bochner
  spaces.
\newblock {\em J. Dynam. Differential Equations}, 2018.

\bibitem{Wang19}
Bixiang Wang.
\newblock Weak pullback attractors for stochastic {N}avier-{S}tokes equations
  with nonlinear diffusion terms.
\newblock {\em Proc. Amer. Math. Soc.}, 147(4):1627--1638, 2019.

\bibitem{LWW18}
Xiaohu Wang, Kening Lu, and Bixiang Wang.
\newblock Wong-{Z}akai approximations and attractors for stochastic
  reaction-diffusion equations on unbounded domains.
\newblock {\em J. Differential Equations}, 264(1):378--424, 2018.

\bibitem{Yang06}
Desheng Yang.
\newblock Random attractors for the stochastic {K}uramoto-{S}ivashinsky
  equation.
\newblock {\em Stoch. Anal. Appl.}, 24(6):1285--1303, 2006.

\bibitem{YK11}
Meihua Yang and P.~E. Kloeden.
\newblock Random attractors for stochastic semi-linear degenerate parabolic
  equations.
\newblock {\em Nonlinear Anal. Real World Appl.}, 12(5):2811--2821, 2011.

\bibitem{YL17}
Jinyan Yin and Yangrong Li.
\newblock Two types of upper semi-continuity of bi-spatial attractors for
  non-autonomous stochastic {$p$}-{L}aplacian equations on {$\Bbb R^n$}.
\newblock {\em Math. Methods Appl. Sci.}, 40(13):4863--4879, 2017.

\bibitem{YY14}
Honglian You and Rong Yuan.
\newblock Random attractor for stochastic partial functional differential
  equations with infinite delay.
\newblock {\em Bull. Korean Math. Soc.}, 51(5):1469--1484, 2014.

\bibitem{DZ10}
Caidi Zhao and Jinqiao Duan.
\newblock Random attractor for the {L}adyzhenskaya model with additive noise.
\newblock {\em J. Math. Anal. Appl.}, 362(1):241--251, 2010.

\bibitem{Zhao17b}
Wenqiang Zhao.
\newblock Long-time random dynamics of stochastic parabolic {$p$}-{L}aplacian
  equations on {$\Bbb{R}^N$}.
\newblock {\em Nonlinear Anal.}, 152:196--219, 2017.

\bibitem{Zhao17}
Wenqiang Zhao.
\newblock Random dynamics of stochastic {$p$}-{L}aplacian equations on
  {$\Bbb{R}^N$} with an unbounded additive noise.
\newblock {\em J. Math. Anal. Appl.}, 455(2):1178--1203, 2017.

\bibitem{ZL11}
Wenqiang Zhao and Yangrong Li.
\newblock Existence of random attractors for a $p$-{L}aplacian-type equation
  with additive noise.
\newblock {\em Abstract and Applied Analysis}, pages 1--21, 2011.

\bibitem{ZZ16}
Shengfan Zhou and Min Zhao.
\newblock Fractal dimension of random attractor for stochastic non-autonomous
  damped wave equation with linear multiplicative white noise.
\newblock {\em Discrete Contin. Dyn. Syst.}, 36(5):2887--2914, 2016.

\bibitem{ZZ17}
Rongchan Zhu and Xiangchan Zhu.
\newblock Random attractor associated with the quasi-geostrophic equation.
\newblock {\em J. Dynam. Differential Equations}, 29(1):289--322, 2017.

\end{thebibliography}

\end{document}